\documentclass{amsart}
\usepackage{marktext}
\usepackage{gimac}

\DeclarePairedDelimiter{\ip}{\langle}{\rangle}

\newtheorem{assumption}[theorem]{Assumption}

\DeclareMathOperator{\Per}{Per}

\DeclareMathOperator{\dist}{dist}

\newcommand{\tmix}{t_\mathrm{mix}}
\newcommand{\tdis}{t_\mathrm{dis}}
\newcommand{\TV}{\mathrm{TV}}

\newcommand{\cyl}{\mathcal{C}}
\newcommand{\zerotuple}{{\bm{0}}}

\newcommand{\Phigh}{\mathcal P_H}
\newcommand{\Plow}{\mathcal P_L}

\begin{document}
\title[Using Bernoulli maps to accelerate mixing]{Using Bernoulli maps to accelerate mixing of a random walk on the torus}

\author[Iyer]{Gautam Iyer}
\address{%
    Department of Mathematical Sciences,
    Carnegie Mellon University,
Pittsburgh, PA 15213.}
\email{gautam@math.cmu.edu}
\author[Lu]{Ethan Lu}
\address{%
    Department of Mathematics,
    Stanford University,
Stanford, CA 94305.}
\email{ethanlu@stanford.edu}
\author[Nolen]{James Nolen}
\address{%
    Department of Mathematics,
    Duke University,
Durham, NC 27708.}
\email{nolen@math.duke.edu}
\thanks{%
  This work has been partially supported by
  the National Science Foundation under
  grants DMS-2108080 and DGE-2146755, 
  and the Center for Nonlinear Analysis.
}
\dedicatory{Dedicated to Robert L. Pego, whose life and work is an inspiration.}
\begin{abstract}
  We study the mixing time of a random walk on the torus, alternated with a Lebesgue measure preserving Bernoulli map.
  Without the Bernoulli map, the mixing time of the random walk alone is $O(1/\epsilon^2)$, where $\epsilon$ is the step size.
  Our main results show that for a class of Bernoulli maps, when the random walk is alternated with the Bernoulli map~$\varphi$ the mixing time becomes $O(\abs{\ln \epsilon})$.
  We also study the \emph{dissipation time} of this process, and obtain~$O(\abs{\ln \epsilon})$ upper and lower bounds with explicit constants.
\end{abstract}
\maketitle

\section{Introduction}

The aim of this paper is to study how certain dynamical systems can accelerate convergence of a random walk to its stationary distribution.
Explicitly, let $\T^d$ denote the $d$-dimensional torus, and consider a (discrete time) dynamical system that preserves Lebesgue measure and whose evolution is determined by a map~$\varphi \colon \T^d \to \T^d$.
Let~$\epsilon > 0$ be small,
and $\epsilon \zeta_n$ be a sequence of i.i.d.\ $\T^d$ valued random variables with mean~$0$ and ``size~$\epsilon$'' (in a sense that will made precise shortly).
We consider the Markov process on the torus defined by
\begin{equation}\label{e:XnTorus}
  X_{n+1} = \varphi(X_n) + \epsilon \zeta_{n+1} \,.
\end{equation}
Throughout this paper, we will assume that the law of the random variables $\epsilon \zeta_{n}$ has a density function $K_\epsilon$.
Under our assumptions, the unique, stationary distribution of the process~$X$ is the Lebesgue measure on $\T^d$, and we will denote this measure by~$\pi$.

We will measure the accelerated mixing of~$X$ by estimating the \emph{mixing time} and the \emph{dissipation time} as~$\epsilon \to 0$.
Recall the \emph{mixing time} of~$X$, denoted by $\tmix(\delta)$, is defined by
\begin{equation}
  \tmix(\delta) = \min \set[\Big]{ n \in \N \st  \norm{ \dist(X_n) - \pi }_\TV \leq \delta  \text{ for all initial distributions of } X_0 }\,.
\end{equation}
Here~$\delta \in (0, 1)$ is any fixed constant, and when~$\delta = 1/2$, we will drop the argument and simply write~$\tmix$ instead of $\tmix(1/2)$.   

If~$\varphi$ is the identity map, then $X$ is simply a random walk on the torus.
In this case, if the distribution of~$\epsilon \zeta_n$ is non-degenerate and regular it is not hard to show
that~$\tmix \approx O(1/\epsilon^2)$ as $\epsilon \to 0$ (see for instance Lemma~1 in~\cite{FannjiangWoowski03}).
For more interesting choices of~$\varphi$, however, the mixing time may be dramatically smaller.

The simplest example of of this is when~$\varphi$ is the one dimensional doubling map:
\begin{equation*}
  \varphi(x) = 2x \pmod{1}\,.
\end{equation*}
In this case it is not hard to see~$\tmix \approx O(\abs{\ln \epsilon})$.
Indeed, the linear structure of~$\varphi$ allows us to write
\begin{equation*}
  X_n = 2^n X_0 + \sum_{k=1}^{n} 2^{n-k} (\epsilon \zeta_k) \pmod{1}\,.
\end{equation*}
If $2^{n-1}\epsilon \approx 1$, then the first term in the above sum is roughly uniformly distributed, and can be used to show $\tmix \approx \abs{\ln_2 \epsilon}$ as~$\epsilon \to 0$.
A similar argument can be used to show that if $\varphi$ is an \emph{ergodic toral automorphism} (see for instance~\S4.2 in~\cite{KatokHasselblatt95}), then $\tmix = O(\abs{\ln \epsilon})$ as $\epsilon \to 0$.
In fact, for both these examples, the distribution of~$X_n$ converges to the stationary distribution at a rate that is faster than exponential (see for instance~\cite{FannjiangWoowski03,FengIyer19}, or Appendix~\ref{s:dexp}, below).

The analysis of both the above examples relies crucially on the linear structure of~$\varphi$.
Even when~$\varphi$ is piecewise affine linear, the analysis breaks down and estimating the mixing time is a much harder problem.
One general result that can be deduced from~\cite{FengIyer19} (see Corollary~\ref{c:expMix}, below), is that when~$\varphi$ is~$C^1$ and generates an \emph{exponentially mixing} dynamical system, and the distribution of~$\epsilon \zeta_k$ is a periodized rescaled non-degenerate Gaussian, then
\begin{equation}\label{e:logEp3}
  \tmix \leq O(\abs{\ln \epsilon}^3)\,.
\end{equation}
This bound, however, is not expected to be sharp.
To briefly explain why, suppose~$X_0$ is concentrated at a point $x \in \T^d$.
After one time step, the noise should ensure that some fraction of the mass of~$X_1$ becomes uniformly spread over the ball~$B(\varphi(x), \epsilon)$.
Since~$\varphi$ generates an exponentially mixing dynamical system, we know that for $n \geq O(\abs{\ln \epsilon})$, a constant fraction of the set~$\varphi^n(B(\varphi(x), \epsilon))$ intersects any given~$\epsilon$ ball.
So, after one more time step, the noise should spread a fraction of the mass uniformly over this~$\epsilon$ ball.
This suggests that the optimal bounds should be $\tmix = O(\abs{\ln \epsilon})$, and not~$O(\abs{\ln \epsilon}^3)$ as stated above.  To turn the above argument into a rigorous proof of~$O(\abs{\ln \epsilon})$ bounds, one needs to control the distance between $X_n$ and  $\varphi^n(X_0)$ for $n = O(\abs{\ln \epsilon})$.
This does not always seem possible, and thus showing~$\tmix \leq O(\abs{\ln \epsilon})$ in this generality is still open.
\medskip

The main purpose of this paper is to improve~\eqref{e:logEp3} and obtain $O(\abs{\ln \epsilon})$ mixing time bounds for a class of exponentially mixing maps with non-degenerate noise.
Roughly speaking, we will need to assume that~$\varphi$ is a piecewise affine linear Bernoulli map.
In dimensions higher than~$1$ we will also require that all cylinder sets are axis-aligned cubes (see Section~\ref{s:notation} for a precise description).
A typical example (when $d = 1$) is the map
\begin{equation*}
  \varphi(x) = \begin{cases}
    3x & x \in [0, \tfrac{1}{3})\,,
    \\
    \frac{3(1 - x)}{2} & x \in [\tfrac{1}{3}, 1)\,.
  \end{cases}
\end{equation*}
In dimension two, another example is shown graphically on the left of Figure~\ref{f:2dCyl}.
The map shown on the right of Figure~\ref{f:2dCyl} does not satisfy our assumptions, as some cylinder sets are rectangles, and not axis-aligned squares.
\begin{figure}[hbt]
  \includegraphics[width=.4\linewidth]{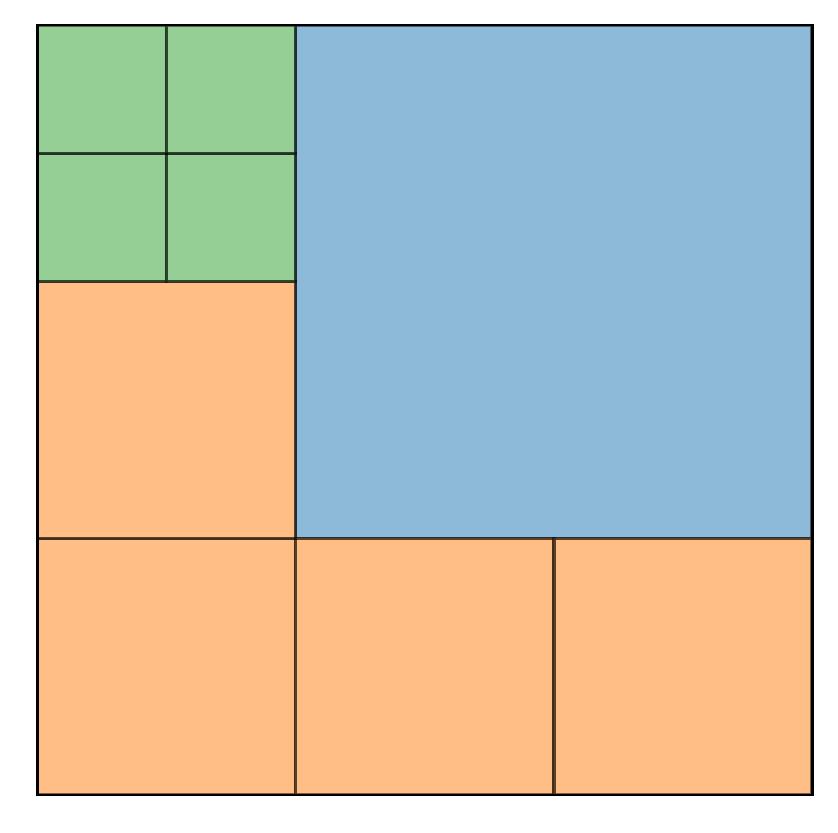}
  \qquad
  \includegraphics[width=.4\linewidth]{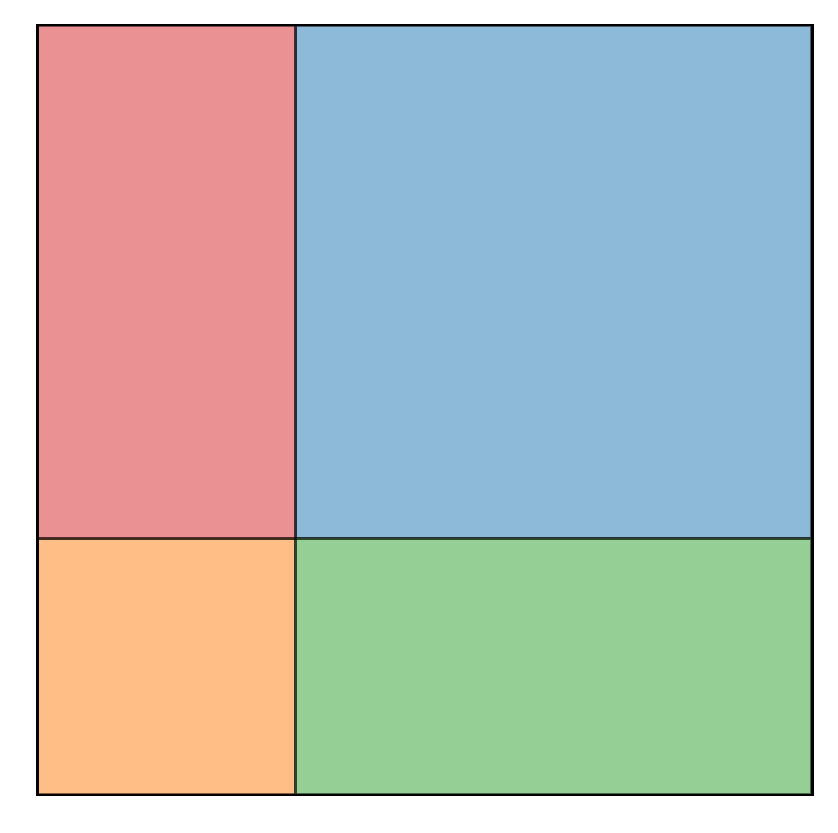}
  \caption{Two examples of the map~$\varphi$ in two dimensions.
    On each smaller region (colored by size), the function~$\varphi$ is an affine linear bijection onto the big square.
  }
  \label{f:2dCyl}
\end{figure}

The first main result in this paper is an $O(\abs{\ln \epsilon})$ upper bound on the mixing time, under certain assumptions on the distribution of~$\epsilon \zeta_n$.
Roughly speaking, we assume that the distribution of $\epsilon \zeta_n$ is either a periodized rescaled non-degenerate Gaussian (Assumption~\ref{a:Gaussian}), or a rescaled separated kernel (Assumption~\ref{a:Separated}).
\begin{theorem}\label{t:TMixUpper}
  Suppose either Assumption~\ref{a:Gaussian} or~\ref{a:Separated} (below) hold.
  Then there exists a constant~$C$ such that for all sufficiently small~$\epsilon$ we have
  \begin{equation}\label{e:tMixGaussian}
    \tmix \leq C \abs{\ln \epsilon}  \,.
  \end{equation}
\end{theorem}

Of course, the above implies for any~$\delta \in (0, 1)$ we have
\begin{equation}
  \tmix(\delta) \leq C \abs{\ln \epsilon} (1 + \abs{\ln \delta})\,.
\end{equation}
While this bound is of the right order in~$\epsilon$, the constant~$C$ is not explicit.
We believe the sharp bound is of the form
\begin{equation}\label{e:TMixDeltaSharpUpper}
  \frac{d \ln \epsilon}{\ln p_{\max}} - C(\delta)
    \leq \tmix(\delta)
    \leq \frac{d \ln \epsilon}{\ln p_{\max}} + C(\delta) \ln( 1 + \abs{\ln \epsilon} )\,,
\end{equation}
where~$p_{\max}$ is the measure of the largest domain on which~$\varphi$ is affine linear, and $C(\delta) < \infty$ is a finite constant that is independent of~$\epsilon$.
Although we are presently unable to prove the upper bound in~\eqref{e:TMixDeltaSharpUpper}, we are able to prove the lower bound, under less restrictive assumptions than that of Theorem~\ref{t:TMixUpper}.

\begin{theorem}\label{t:TMixLower}
  Suppose Assumption~\ref{a:Cubes} holds, and for all sufficiently small~$\epsilon > 0$ we have
  \begin{equation}\label{e:FirstMomentKep}
    \int_{\T^d} d(y, 0) K_\epsilon(y) \, \pi(dy) \leq \epsilon\,.
  \end{equation}
  (Here $d(y, 0)$ denotes the distance between~$y$ and $0$ on the torus $\T^d$.)
  Then for every~$\delta \in (0, 1)$ the mixing time of~\eqref{e:XnTorus} is bounded below by
  \begin{equation}\label{e:TMixLower}
    \frac{d \ln \epsilon}{\ln p_{\max}} - C(1 + \abs{\ln(1 - \delta)})
      \leq \tmix(\delta)\,,
  \end{equation}
  for some constant $C$ and all sufficiently small~$\epsilon > 0$.
\end{theorem}

Thus far we have been discussing the mixing time $t_{mix}$, which quantifies convergence to the stationary distribution in terms of total variation.
On the other hand, in many physical contexts (such as passive scalar advection~\cite{
  Obukhov49,
  Corrsin51,
  Pierrehumbert94,
  ShraimanSiggia00,
  HaynesVanneste05,
  Thiffeault12,
  MilesDoering18
}), $L^2$ convergence is more natural.
The \emph{dissipation time} (see~\cite{FannjiangWoowski03,FannjiangNonnenmacherEA04,FengIyer19}) measures the rate of convergence of the distribution of $X_n$ to~$\pi$ in $L^2$, for initial distributions which are also~$L^2$.
Explicitly, for
any $\delta \in (0, 1)$, define the dissipation time~$\tdis(\delta)$ by
\begin{equation*}
  \tdis(\delta) = \min \set[\Big]{ n \in \N \st \norm{\theta_n}_{L^2} \leq \delta \norm{\theta_0}_{L^2} \text{ for all } \theta_0 \in \dot L^2 }\,.
\end{equation*}
Here~$\theta_n(x) = \E^x \theta_0(X_n) = \E(\theta_0(X_n) \given X_0 = x)$, and $\dot L^2 \subseteq L^2(\T^d, \pi)$ is the sub-space of all mean-zero functions.
For notational convenience,
when~$\delta = 1/2$, we will drop the argument and write $\tdis$ instead of $\tdis\paren{1/2}$.

Before stating our result concerning the dissipation time of~\eqref{e:XnTorus}, we mention that the dissipation time and mixing time are related in a general setting (see also~\cite{IyerZhou22}).
\begin{proposition}\label{p:tmixTdis}
  For any Markov process, the dissipation time is bounded in terms of the mixing time by the inequality
  \begin{equation}\label{e:tdis-leq-tmix}
    \tdis(\delta) \leq \tmix\paren[\Big]{ \frac{\delta^2}{4} }
      \leq \paren[\big]{2 - \log_\delta 4} \tmix(\delta)
    \,.
  \end{equation}
  For the Markov process~\eqref{e:XnTorus}, if the densities $K_\epsilon$ satisfy
  \begin{equation}\label{e:KepL2}
    \sup_{\epsilon > 0} \epsilon^{d/2} \norm{K_\epsilon}_{L^2} = \bm K < \infty\,,
  \end{equation}
  then for any~$\delta, \delta' \in (0, 1)$ we have
  \begin{equation}\label{e:tmix-leq-tdis}
    \tmix(\delta') \leq 2 + \log_\delta \paren[\Big]{ \frac{\delta' \epsilon^{d/2}}{\bm K} } \tdis(\delta)\,.
  \end{equation}
\end{proposition}

The~$L^2$ convergence of systems in the form~\eqref{e:XnTorus} have been studied by many authors~\cite{ThiffeaultChildress03,FannjiangWoowski03,FannjiangNonnenmacherEA04,FannjiangNonnenmacherEA06,FengIyer19,OakleyThiffeaultEA21}.
The work of Fannjiang \etal~\cite{FannjiangWoowski03,FannjiangNonnenmacherEA04,FannjiangNonnenmacherEA06} studies the dissipation time under certain spectral assumptions on the associated Koopman operator.
These assumptions, however, are hard to verify.

Previous work of Feng and the first author~\cite{FengIyer19} used a Fourier splitting method to show that if~$\varphi$ is~$C^1$ and generates an exponentially mixing dynamical system, then the dissipation time of the Markov process~\eqref{e:XnTorus} is bounded by
\begin{equation}\label{e:TDisLep2}
  \tdis \leq C \abs{\ln \epsilon}^2\,,
\end{equation}
for some finite constant~$C$ and all sufficiently small~$\epsilon > 0$.
An immediate consequence of this is the mixing time bound~\eqref{e:logEp3} mentioned above.

\begin{corollary}\label{c:expMix}
  Suppose~$K_\epsilon$ is a periodized rescaled standard Gaussian (i.e.\ $K_\epsilon$ is defined by~\eqref{e:KepsPeriodized}, below, with $\check K$ being a standard Gaussian in $\R^d$).
  If $\varphi \colon \T^d \to \T^d$ is~$C^1$ and generates an exponentially mixing, dynamical system, then there exists a constant~$C$ such that~\eqref{e:logEp3} holds for all sufficiently small~$\epsilon$.
\end{corollary}
\begin{proof}
  Corollary 2.5 in~\cite{FengIyer19} proves~\eqref{e:TDisLep2}, and now Proposition~\ref{p:tmixTdis} implies~\eqref{e:logEp3}.
\end{proof}
\smallskip

We now return to studying the process~\eqref{e:XnTorus} when~$\varphi$ is the piecewise affine linear Bernoulli map described earlier.
Such maps generate exponentially mixing dynamical systems (see for instance~\cite{KatokHasselblatt95,SturmanOttinoEA06}), and hence~\cite{FengIyer19} guarantees the $O(\abs{\ln \epsilon}^2)$ upper bound~\eqref{e:TDisLep2}.
Of course, if~$K_\epsilon$ satisfies the assumptions in Theorem~\ref{t:TMixUpper}, then Proposition~\ref{p:tmixTdis} will imply the stronger upper bound
\begin{equation}\label{e:TDis3CLep}
  \tdis \leq \tmix\paren[\Big]{\frac{1}{16}} \leq 4 \tmix \leq 4C \abs{\ln \epsilon}\,,
\end{equation}
for some constant~$C > 0$.
We will obtain now improve~\eqref{e:TDis3CLep} by obtaining the exact constant on the right, and making less restrictive assumptions on~$K_\epsilon$.

\begin{theorem}\label{t:pl-dis-time}
  Suppose Assumption~\ref{a:Cubes} holds and $\supp( K_\epsilon ) \subseteq B(0, \epsilon)$.
  Moreover, assume that the Fourier coefficients of~$K_\epsilon$ satisfy
  \begin{equation}\label{e:KHatDecay}
    \sup_{\abs{k} > \frac{c}{\epsilon}} \abs{\hat K_\epsilon(k)} < 1\,,
  \end{equation}
  for any~$c > 0$ and all sufficiently small~$\epsilon > 0$.
  Then, there exists an explicit constant~$C$, such that for every $\delta \in (0, 1)$,
  \begin{equation}\label{e:plDisTime}
      \frac{d \ln \epsilon}{\ln p_{\min}}
	- C ( 1 + \abs{\ln (1 - \delta)} )
      \leq \tdis(\delta)
      \leq \frac{d \ln \epsilon}{\ln p_{\max}}
	+ C\paren[\Big]{ 1 + \frac{\abs{\ln \delta}}{\delta^2} }
  \end{equation}
  for all sufficiently small~$\epsilon > 0$.
  Here $p_{\min}$ is the measure of the smallest domain on which~$\varphi$ is affine linear.
\end{theorem}
\begin{remark}
  Since~$\tdis(\delta)$ is a decreasing function of~$\delta$, the lower bound in~\eqref{e:plDisTime} is useful when~$\delta$ is close to~$1$, and the upper bound is useful when~$\delta$ is close to~$0$.
\end{remark}

In general,~$p_{\min} < p_{\max}$, and so there may be a gap between the upper and lower bounds in ~\eqref{e:plDisTime}.
However, in the symmetric case (when each of the regions on which~$\varphi$ is affine linear have equal size), the upper and lower bounds in~\eqref{e:plDisTime} match.
Obtaining matching upper and lower bounds for the mixing time is indicative of a \emph{cutoff phenomenon}~\cite{Diaconis96,LevinPeresEA09,BasuHermonEA17}.
Although we believe the cutoff phenomenon happens even when~$p_{\min} < p_{\max}$, we are presently unable to prove it.

\begin{comment}[Explain the following]
  \begin{enumerate}[(1)]
    \item The multiplicative constants in~\eqref{e:plDisTime} are explicit, but they don't match.
    \item Doesn't quite give a cutoff phenomenon for the dissipation time.
      But does show
      \begin{equation*}
	\lim_{\epsilon \to 0} d_2\paren[\Big]{ (1 - \delta) \frac{d \ln \epsilon}{\ln p_{\min}}} = 1\,,
	\quad
	\lim_{\epsilon \to 0} d_2\paren[\Big]{ (1 + \delta) \frac{d \ln \epsilon}{\ln p_{\max}}} = 0\,,
      \end{equation*}
      where
      \begin{equation*}
	d_t(n) = \sup\set{ \norm{\theta_n}_{L^2} \st \theta_n(x) = \E^x \theta_0\,,~
	\norm{\theta_0}_{L^2} = 0\,,
	\quad\text{and} \quad \theta_0 \in \dot L^2\,,}
      \end{equation*}
  \end{enumerate}
\end{comment}
\medskip

The main idea behind our proofs
is to find a sufficiently large family of initial distributions which are both mixed well by~$\varphi$, and not perturbed too much by the noise.
For Theorem~\ref{t:TMixUpper} the family we use are bump functions on ``cylinder sets'' (which will be defined shortly).
These bump functions will vanish linearly near the boundary, giving a good estimate on the rate at which~$X$ leaves cylinder sets.
The key to using this is an eigenfunction like inequality~\eqref{e:KepsFs}, explained in Section~\ref{s:MixTime}, below.

For Theorems~\ref{t:TMixLower} and~\ref{t:pl-dis-time} the family of initial distributions we use are piecewise constant on cylinder sets.
This family of distributions are better mixed by the deterministic dynamics of~$\varphi$, which is what leads to the explicit coefficients of $\abs{\ln \epsilon}$ coefficients in~\eqref{e:plDisTime}.
The drawback of this method is that to apply it one needs to first approximate the initial distribution by one which is piecewise constant to a high degree of accuracy.
Using Fourier projections, we are able to obtain convergence of the density of~$X$ in~$L^2$, but not in~$L^1$.
This allows us to obtain the dissipation time bounds~\eqref{e:plDisTime}, but not the mixing time bound~\eqref{e:tMixGaussian}.
In order to obtain a bound on the mixing time by this method, one would need to obtain suitable $W^{1,1}$ decay of the density, which we are presently unable to do.
\medskip

There are several elementarily stated variants of our main results that \emph{can not} be proved by our methods, and require new ideas.
\begin{enumerate}
  \item
    What is the mixing time if the cylinder sets are cuboids, and not cubes (e.g.\ the map~$\varphi$ shown on the right of Figure~\ref{f:2dCyl})?
    More generally, what is the mixing time if the map~$\varphi$ is Bernoulli and the cylinder sets ``shrink nicely''?
    This is particularly interesting in light of a recent striking result~\cite{DolgopyatKanigowskiEA21} which shows that any~$C^{1+\alpha}$ exponentially mixing map is Bernoulli.
    For general Bernoulli systems the cylinder sets could be very irregular.
    Our methods can not be used to address this situation and new ideas are required.
  \item
    If the map~$\varphi$ is a bi-infinite Bernoulli shift (such as the bakers map~\cite{SturmanOttinoEA06,KatokHasselblatt95}, or the folded bakers map appearing in~\cite{ElgindiZlatos19}), then is the mixing time still~$O(\abs{\ln \epsilon})$?
    Numerical evidence suggests this is true, but we are presently unable to generalize our methods to this situation.
  \item
    Does the Markov process~$X$ exhibit a cutoff phenomenon?
    When the map~$\varphi$ is a uniformly expanding map, this is easy to show.
    When $p_{\max} = p_{\min}$, the bound~\eqref{e:plDisTime} shows a cutoff phenomenon in the~$L^2$ sense, but not necessarily in the usual total variation sense.
    In general, when~$p_{\max} \neq p_{\min}$, and~$\varphi$ is only piecewise affine linear, we suspect~$X$ exhibits a cutoff phenomenon, but are presently unable to prove it.
  \item 
    In the spatially discrete, speeding up convergence of Markov chains by various techniques has been extensively studied~\cite{
      ChungDiaconisEA87,
      ChenLovaszEA99,
      DiaconisGraham92,
      DiaconisHolmesEA00,
      Neal04,
      KapferKrauth17,
      ConchonKerjan22,
      Diaconis13}.
    A recent result that of particular relevance is that of Chatterjee and Diaconis~\cite{ChatterjeeDiaconis20}, where the authors study~\eqref{e:XnTorus} on the discrete torus, and show that choosing~$\varphi$ to be almost any bijection on the state space results in an exponential speedup.
    To the best of our knowledge, there is no continuous version of this result and the method used in~\cite{ChatterjeeDiaconis20} can not easily be adapted to the continuous setting.
    %Is there a continuous space analog of this result?
    %
\end{enumerate}

Finally, we mention related questions in the continuous time setting.
In the context of fluid dynamics, there has been a lot of recent activity studying the notion and applications of \emph{enhanced dissipation} (see for instance~\cite{
  ConstantinKiselevEA08,
  Zlatos10,
  BedrossianCotiZelati17,
  FengIyer19,
  Wei19,
  CotiZelatiDelgadinoEA20,
  FengFengEA20,
  BedrossianBlumenthalEA21,
  ColomboCotiZelatiEA21,
  IyerXuEA21,
  FengMazzucato22,
  IyerZhou22,
  Seis22
}).
This notion measures faster convergence of the associated PDE to its equilibrium distribution.
The motivation for the present paper was to provide a different perspective and study this phenomenon using convergence of Markov processes.
While we are able to obtain optimal bounds in one discrete time setting, there are many related continuous time settings where the optimal bounds are open.

\subsection*{Plan of this paper}
In Section~\ref{s:notation} we describe~$\varphi$, and state the assumptions required for our main results precisely.
In Section~\ref{s:MixTime} we prove Theorem~\ref{t:TMixUpper} by constructing a family of bump functions supported on cylinder sets.
In Section~\ref{s:MixTimeLower} we prove Theorem~\ref{t:TMixLower} by estimating the rate at which piecewise constant functions get mixed.
In Section~\ref{s:TMixTDis} we prove Proposition~\ref{p:tmixTdis} relating the mixing time and dissipation time, and in Section~\ref{s:dissipationTimeBounds} we prove Theorem~\ref{t:pl-dis-time} obtaining upper and lower bounds on~$\tdis$.
Finally, in Appendix~\ref{s:dexp} we show that when~$\varphi$ is a linear expanding map (or an ergodic toral automorphism) we show that the convergence of~$X_n$ happens at a double exponential rate.

\section{Notation and Preliminaries}\label{s:notation}
\subsection{Piecewise affine linear Bernoulli maps.}\label{s:Phi}
We begin by precisely describing the map~$\varphi:\T^d \to \T^d$ used in \eqref{e:XnTorus}.
Let~$M \geq 2$ and $\check E_1, \dots, \check E_M \subseteq \R^d$ be a partition of $[0, 1)^d$.
When $d = 1$ we assume each $\check E_i$ is an interval, and when $d > 1$ we assume 
each $\check E_i$ is a $d$-dimensional cube with sides parallel to the coordinate axes (as in the left figure in Figure~\ref{f:2dCyl}).

Let $p_i = \check \pi(\check E_i) \in (0,1)$, where $\check \pi$ denotes the Lebesgue measure on $\R^d$.
We will assume~$\check \varphi_i \colon \check E_i \to [0, 1)^d$ are affine linear bijections of the form
\begin{equation}\label{e:PhiDef}
  \check \varphi_i(\check x) = \frac{\check D_i \check x}{p_i^{1/d}} + \check e_i\,,
\end{equation}
for some vectors~$\check e_i \in \R^d$, and orthogonal matrices~$\check D_i$.

We now project onto the torus $\T^d$, and define the map~$\varphi \colon \T^d \to \T^d$.
Let $\Pi \colon \R^d \to \T^d = \R^d / \Z^d$ denote the canonical projection, $E_i = \Pi(\check E_i)$, and let
\begin{equation*}
  \varphi_i(x) = \Pi \circ \check \varphi_i( \check x )\,,
\end{equation*}
where $\check x \in [0, 1)^d$ is the unique point such that $\Pi(\check x) = x \in \T^d$.
Define the expanding map~$\varphi \colon \T^d \to \T^d$ by
\begin{equation}\label{e:phiExpanding}
  \varphi(x) = \varphi_i(x)\,,
  \quad\text{if }x \in E_i \,.
\end{equation}
With this notation, we define quantities~$p_{\min}$ and $p_{\max}$ (appearing in Theorems~\ref{t:TMixLower} and~\ref{t:pl-dis-time}) by
\begin{equation*}
  p_{\min} \defeq \min_{i \in \mathcal I} \pi(E_i)\,,
  \quad\text{and}\quad
  p_{\max} \defeq \max_{i \in \mathcal I} \pi(E_i)\,.
\end{equation*}

Note that
\begin{equation*}
  \sum_{i \in \mathcal I} p_i = \sum_{i \in \mathcal I} \pi( E_i ) = \pi(\T^d) = 1\,.
\end{equation*}
Hence, for any Borel set~$A \subseteq \T^d$,
\begin{equation*}
  \pi(\varphi^{-1}(A))
    = \sum_{i=1}^M \pi(\varphi_i^{-1}(A))
    = \pi(A) \sum_{i=1}^M p_i
    = \pi(A)\,,
\end{equation*}
which shows that~$\varphi$ preserves the Lebesgue measure~$\pi$.
This implies that the unique stationary distribution of the process~\eqref{e:XnTorus} is also the Lebesgue measure~$\pi$.

\subsection{Cylinder sets and the Bernoulli shift.}\label{s:Cylinder}
Our analysis of $X_n$ relies on the fact that the map~$\varphi$ has the structure of a \emph{Bernoulli shift} on the space of one-sided sequences (see for instance~\cite{KatokHasselblatt95,SturmanOttinoEA06}).
The building block for functions that are controllably mixed will be based on \emph{cylinder sets}, which we define in this section.

Let $\mathcal I = \set{1, \dots, M}$, and $\mathcal T$ denote the set of all finite length $\mathcal I$-valued tuples.
Explicitly,
\begin{equation*}
    \mathcal T
    = \set{\zerotuple} \cup \bigcup_{m = 1}^\infty \mathcal I^m,
\end{equation*}
where $\zerotuple$ denotes the empty tuple.
Given a tuple $s = (s_0, \dots, s_{m-1}) \in \mathcal T$ we use $\abs{s} = m$ to denote the length of the tuple~$s$, with \(\abs{\zerotuple} = 0\) by convention.

Let \(\sigma\colon \mathcal T \to \mathcal T\) be the \emph{Bernoulli left shift}.
That is, $\sigma(s)$ removes the first coordinate of~$s$ and shifts the other coordinates left.
More precisely, we define
\[
    \sigma(s_0, \ldots, s_{m-1}) = (s_1, \ldots, s_{m-1}) \,,
    \quad\text{and}\quad
    \sigma(\zerotuple) = \zerotuple\,.
\]
Let~$\sigma^k$ denote the \(k\)-fold composition of the map \(\sigma\).

\begin{figure}[bt]
    \includegraphics[width=.3\linewidth]{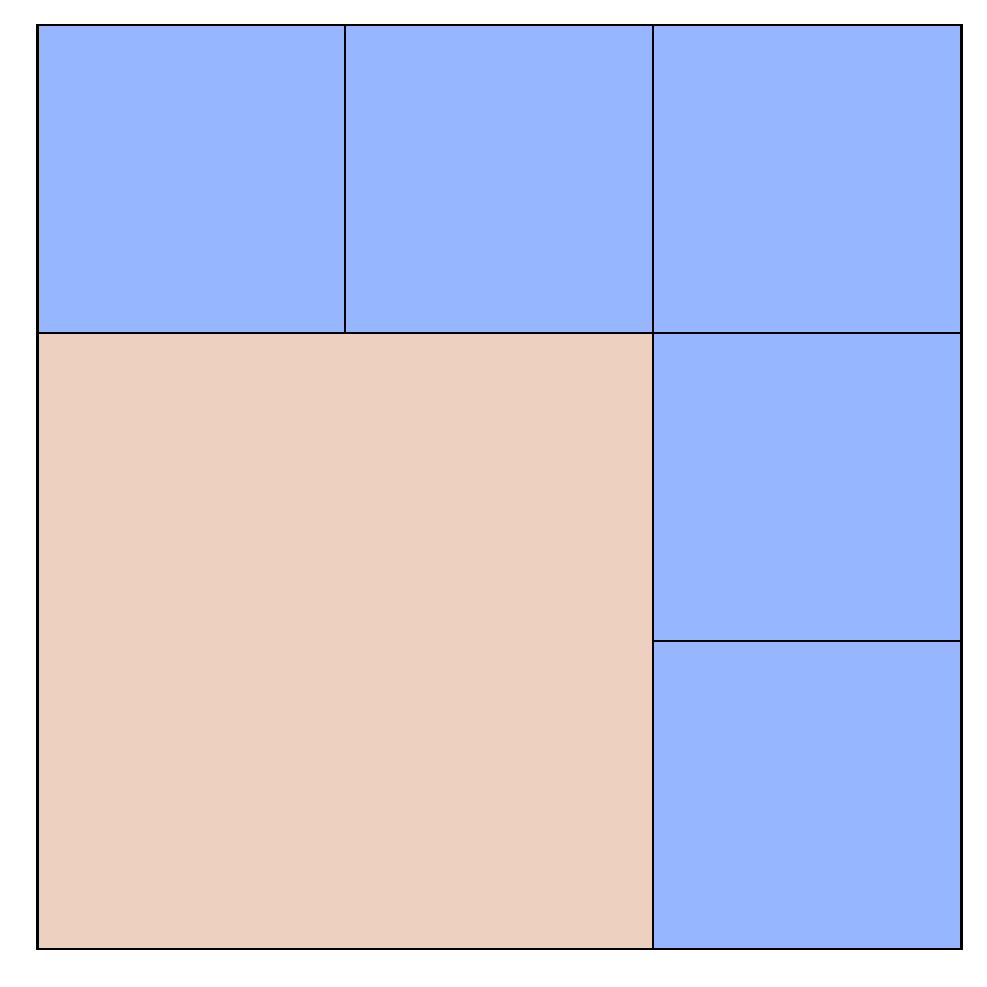}
    \includegraphics[width=.3\linewidth]{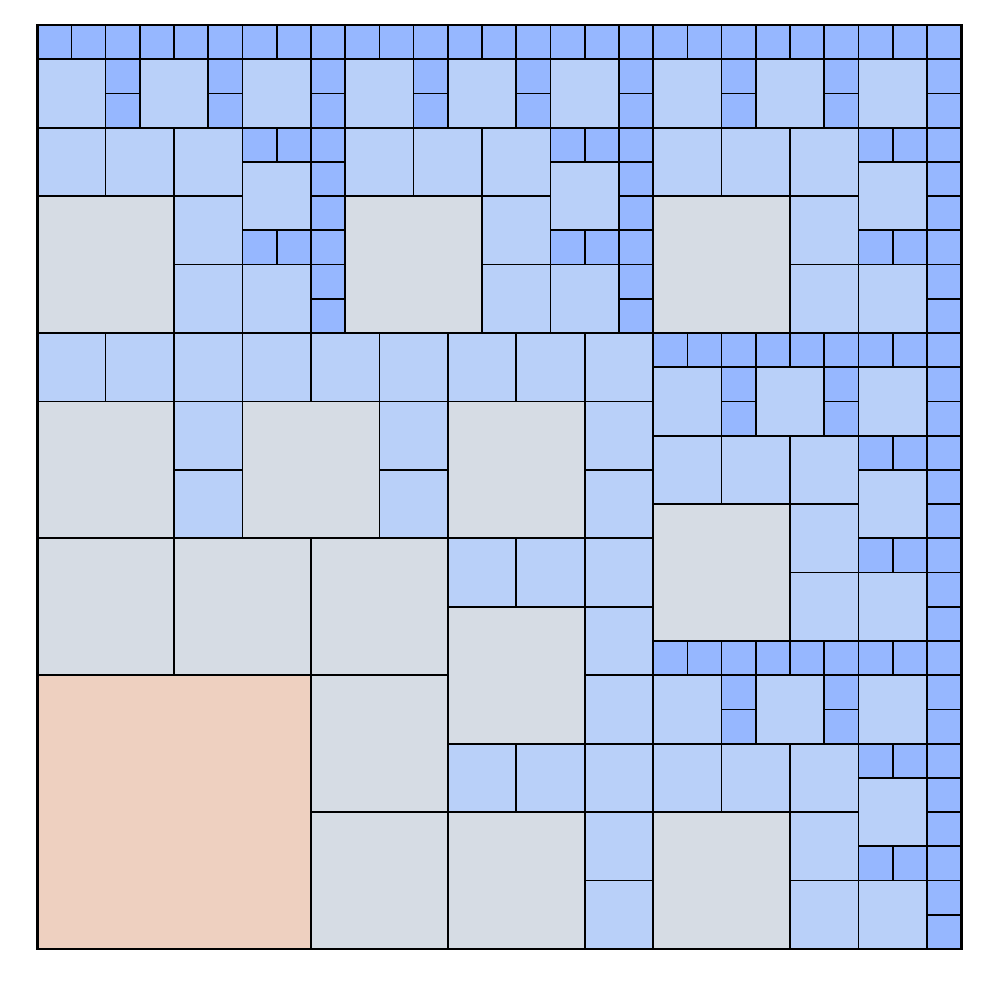}
    \includegraphics[width=.3\linewidth]{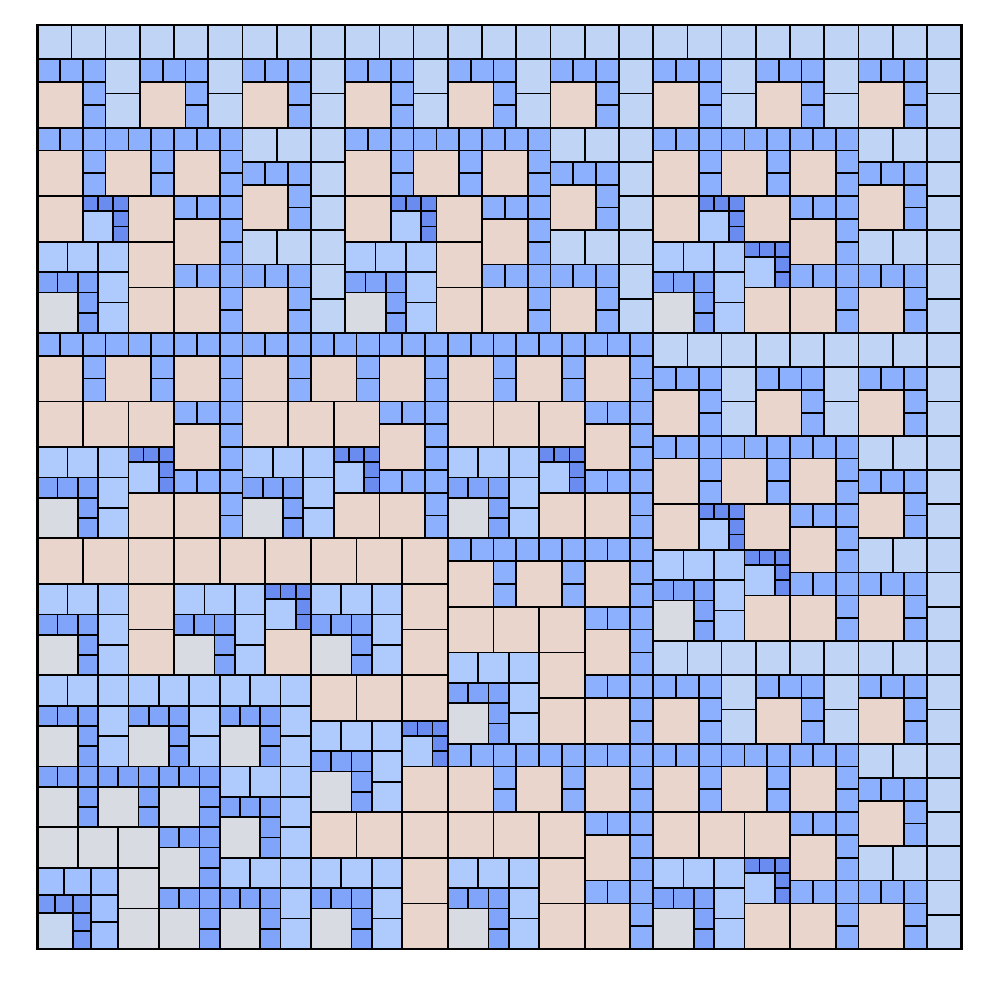}
  \qquad
  \caption{A few examples of cylinder sets.
    The leftmost figure shows order~$0$ cylinder sets, which are simply the domains~$E_1, \dots, E_6$.
    The middle figure shows all cylinder sets of order~$2$.
    The rightmost figure shows a partition of the torus into cylinder sets of different orders, that will be used in the proof.
  }
  \label{f:CylSubdivisionFigure}
\end{figure}
Now given a map \(\varphi\) as in~\eqref{e:phiExpanding}, we can define the associated cylinder set associated to a tuple~\(s \in \mathcal T\) by
\begin{equation}\label{e:CylS}
  \cyl_s \defeq \set{ x \in \T^d \st \varphi^n(x) \in  E_{s_n} \text{ for all } n \leq \abs{s} }\,.
\end{equation}
When~$s = \zerotuple$, the associated cylinder set $\cyl_s$ is the whole torus $\T^d$.
When $s = (s_0)$ is a tuple of length~$1$, the associated cylinder set~$\cyl_s$ is simply the domain~$E_{s_0}$ (see for instance the leftmost figure in Figure~\ref{f:CylSubdivisionFigure}).
When~$\abs{s} > 1$, each of these cylinder sets get subdivided further, forming finer and finer partitions of the torus (e.g.\ the middle figure in Figure~\ref{f:CylSubdivisionFigure}).

Note that the action of~$\varphi$ on cylinder sets is simply the Bernoulli shift.
That is, for any~$s \in \mathcal T$,
\begin{equation*}
  \varphi(\cyl_s) = \cyl_{\sigma s}\,.
\end{equation*}
In particular, this means that~$\varphi^{\abs{s}}(\cyl_s) = \T^d$, and so an initial distribution that is supported on~$\cyl_s$ becomes spread over~$\T^d$ after~$\abs{s}$ iterations of~$\varphi$.
This can be used to show that~$\varphi$ generates an \emph{exponentially mixing} dynamical system (see for instance~\cite{KatokHasselblatt95,SturmanOttinoEA06}).

Since the process~$X$ is constructed by intertwining the action of~$\varphi$ with noise, this suggests that if~$X_0$ is concentrated on one cylinder set~$\cyl_s$, then the distribution of $X_{\abs{s}}$ should be spread out over the whole torus.
This, however, is not easy to prove as the action of the noise does not necessarily commute with the dynamics of~$\varphi$.
The main idea behind the proof of Theorem~\ref{t:TMixUpper} is to construct a special distribution that is supported on~$\cyl_s$ and is provably mixed after~$\abs{s}$ iterations.
We do this in Section~\ref{s:GaussianMixTime}, below.

\subsection{Assumptions on \texorpdfstring{$\varphi$}{phi} and the noise.}

We now precisely describe the assumptions that are required for our results.
The first assumption is on the geometry of the cylinder sets.
\begin{assumption}\label{a:Cubes}
  Assume that the following hold:
  \begin{enumerate}[(1)]
    \item Each~$\varphi_i \colon E_i \to \T^d$ (defined in~\eqref{e:phiExpanding}) is a bijection.
    \item\label{i:mixing}
      For every~$\check x \in \partial [0, 1]^d$, there exists~$i \in \mathcal I$ and~$\check y \in (0, 1)^d$ such that~$\check \varphi_i(\check y) = \check x$.
    \item If~$d \geq 2$, every cylinder set~$\cyl_s$ (as defined in~\eqref{e:CylS}) is an axis-aligned cube.
  \end{enumerate}
\end{assumption}

If the cylinder sets are not exactly cubes, Theorems~\ref{e:TMixLower} and~\ref{t:pl-dis-time} can still be used provided the \emph{volume to perimeter} ratio of cylinder sets is controlled by the diameter.
Specifically, the quantity that needs to stay bounded is the right hand side of~\eqref{e:fullCylLpEstimate} that appears in Lemma~\ref{l:pcMix} from Section~\ref{s:MixTimeLower} below.
This condition, however, is hard to verify for general Bernoulli systems, and so we restrict our attention to the piecewise affine linear case in this paper.

The assumptions on the noise distribution in Theorem~\ref{t:TMixUpper} are a little more restrictive.
The noise has to either be a periodized, non-degenerate Gaussian, or a sum of separated kernels with controlled first moments.
We state this as our next two assumptions.

\begin{assumption}\label{a:Gaussian}
  In addition to Assumption~\ref{a:Cubes}, suppose the following hold:
  \begin{enumerate}[(1)]\reqnomode
    \item
      The function~$\check K$ is the density of a Gaussian in $\R^d$ with mean~$0$ and non-degenerate covariance matrix, and the densities~$K_\epsilon$ are obtained by rescaling and periodizing~$\check K$.
      Explicitly,~$K_\epsilon$ is defined by
      \begin{equation}\label{e:KepsPeriodized}
	K_\epsilon(x) = \sum_{n \in \Z^d} \check K_\epsilon(\check x + n)\,,
	\quad
	\check K_\epsilon (\check x) = \frac{1}{\epsilon^d} \check K\paren[\Big]{ \frac{\check x}{\epsilon} }
	\,.
      \end{equation}
    \item 
      The function $\check K$ is invariant under the action of each orthogonal matrix $\check D_i$ in~\eqref{e:PhiDef}.
      That is, we assume
      \begin{equation}\label{e:KSymm}
	\check K( \check D_i \check x ) = \check K( x )
	\quad \text{for all } \check x \in \R^d\,.
      \end{equation}
  \end{enumerate}
\end{assumption} 

\begin{remark*}
  The invariance assumption~\eqref{e:KSymm} is satisfied if each of the orthogonal matrices~$\check D_i$ commute with the covariance matrix of~$\check K$.
\end{remark*}

\begin{assumption}\label{a:Separated}
  In addition to Assumption~\ref{a:Cubes}, suppose the following hold:
  \begin{enumerate}[(1)]\reqnomode
    \item
      The distributions~$K_\epsilon$ are obtained by rescaling and periodizing a distribution~$\check K$ (as in~\eqref{e:KepsPeriodized}), and~$\check K$ satisfies the conditions below.
    \item 
      For every~$\eta > 0$ we have
      \begin{equation}
	\inf_{\check x, \check y \in [0, \eta)^d} \check K(\check x - \check y) > 0\,.
      \end{equation}
    \item\label{i:sep}
      There exists a family of densities~$\set{ \check K_n }_{n \geq 1}$ such that~$\check K_n \to \check K$ in~$L^1(\R^d)$, and each~$\check K_n$ is in the form
      \begin{equation}\label{e:Ksep3}
	\check K_n(\check x) = \sum_{i = 1}^n b_{n, i} \prod_{j=1}^d \check k_{n,i}( \check x_j )\,,
      \end{equation}
      for some even functions $\check k_{n, i}:\R \to \R$, and constants $b_{n, i} \geq 0$.
      Moreover, there exists a constant $A > 0$ such that such that for every~$n, i$ we have
      \begin{equation}\label{e:Ksep3Moments}
	\int_{\R} \check k_{n,i}(\check x_j) \, d\check x_j = 1\,,
	\quad
	\int_{\R} \abs{\check x_i} \check k_{n,i}(\check x_j) \, d\check x_j \leq A\,,
	\quad
	\int_0^{\frac{1}{2}} \check x_i \check k_{n,i}(\check x_j) \, d\check x_j \geq \frac{1}{A}\,.
      \end{equation}
    \item
      The density~$\check K$ is invariant under the action of each~$\check D_i$ as in~\eqref{e:KSymm}.
  \end{enumerate}
\end{assumption} 

We do not presently have a simple description of the class of probability densities that satisfy condition~\ref{i:sep} in Assumption~\ref{a:Separated}.
If~\eqref{e:Ksep3Moments} did not require a \emph{lower bound} on the first moments, then any compactly supported~$L^1$ probability distribution can be expressed as the~$L^1$ limit of distributions~$\check K_n$ in the form~\eqref{e:Ksep3}.
The lower bound, however, breaks the standard approximation arguments.

\section{Upper bounds on the Mixing Time.}\label{s:MixTime}
\subsection{Proof of Theorem~\ref{t:TMixUpper}}\label{s:GaussianMixTime}
As mentioned above, the main idea behind the proof of Theorem~\ref{t:TMixUpper} is to construct a family of ``bump functions'' supported on cylinder sets whose behavior is controlled under the evolution of~$X$.

To state this, it is convenient to introduce the operators~$T_*$ and~$U_*$.
The operator~$T_*$ is the push forward of a distribution by the transition kernel of~$X$.
Explicitly, we define
\begin{equation}\label{e:TStarDef}
  T_* \mu(y) \defeq \int_{\T^d} \mu(dx) \rho(x, y)\,,
\end{equation}
where $\rho(x, y)$ is the transition density of the process~$X$, and~$\mu$ is a finite measure.  If $\mu <\!\!< \pi$ and $\frac{d \mu}{d \pi} = f \in L^1(\pi)$, we define the action of $T_*$ on the density $f$ by
\begin{gather}\label{e:TStarDefPrime}\tag{\ref*{e:TStarDef}$'$}
  T_* f(y) = \int_{\T^d} \pi(dx) f(x) \rho(x, y) \,.
\end{gather}
From~\eqref{e:XnTorus} we note that we may also write
\begin{equation}\label{e:DefUTStar}
  T_* f = K_\epsilon * U_*f\,,
  \qquad\text{where}\qquad 
  U_*f \defeq \sum_{i = 1}^M f \circ \varphi_i^{-1} \abs{ \det( D \varphi_i^{-1} ) } \,.
\end{equation}

Recall by Assumption~\ref{a:Cubes}, all cylinder sets are intervals for~$d = 1$ and axis-aligned cubes for $d > 1$.
We will use~$\ell_s$ denote the length of the interval~$\cyl_s$ when~$d = 1$, and the side length of the cube~$\cyl_s$ when $d > 1$.
For convenience define~$\lambda_s = 1/\ell_s$.
Explicitly,
\begin{equation}\label{elldef}
  \ell_s = \pi(\cyl_s)^{1/d}\,,
  \quad\text{and}\quad
  \lambda_s = \frac{1}{\ell_s} = \frac{1}{\pi(\cyl_s)^{1/d}}\,.
\end{equation}

\begin{lemma}\label{l:FsExist}
  Suppose either Assumption~\ref{a:Gaussian} or Assumption~\ref{a:Separated} hold.
  Then there exist a family of continuous functions~$\set{\bm F_s \st s \in \mathcal T}$ with the following properties.
  Each function~$\bm F_s$ is supported on~$\cyl_s$, is strictly positive on the interior of~$\cyl_s$, is normalized so that~$\norm{\bm F_s}_{L^1} = 1$, and satisfies the identity
  \begin{equation}\label{e:UStarFs}
    U_* \bm F_s = \bm F_{\sigma s}\,,
    \quad\text{for all } s \in \mathcal T - \set{\zerotuple}\,.
  \end{equation}
  Moreover, there exists $a < \infty$, $\gamma > 0$ such that for all~$s \in \mathcal T$ we have
  \begin{equation}\label{e:KepsFs}
    K_\epsilon * \bm F_s
      \geq \paren{ 1 - a (\lambda_s \epsilon)^\gamma} \bm F_s\,.
  \end{equation}
\end{lemma}

When~$K_\epsilon$ is a periodized rescaled Gaussian, one can simply choose~$\bm F_s$ to be the principal eigenfunctions of an elliptic operator (see the proof in Section~\ref{s:FsExist}, below).
In the other case the construction is more involved and is presented in Section~\ref{s:F0Exist}, below.

The key to proving Theorem~\ref{t:TMixUpper} is to show that if the density of the distribution of $X_0$ is $\bm F_s$, then after time $\abs{s}+1$ the density of the distribution of $X_{\abs{s}+1}$ is bounded below, away from~$0$.
This is our next lemma.

\begin{lemma}\label{l:TstarLower}
  Let~$a, \gamma$ be as in Lemma~\ref{l:FsExist} and set~$\eta = (2a)^{1/\gamma}$.
There exist constants $\beta, \beta' \in (0,1]$ such that for all $\epsilon > 0$ and all $s \in \mathcal T$ such that $\ell_s \geq \eta \epsilon$, we have
  \begin{equation}\label{e:TStarLower}
    T_*^{n} \bm F_s
      \geq \begin{cases}
	\beta \bm F_{\sigma^n s}, & \quad \forall\; n \leq \abs{s}\,,
	\\
	\beta', & \quad \forall \;n \geq \abs{s} + 1\, .
      \end{cases}
  \end{equation}
\end{lemma}

By bounding the density of $X_1$ below by a combination of the functions $\bm F_s$ above, we claim that any initial distribution becomes bounded away from~$0$ in time $O(\abs{\ln \epsilon})$.

\begin{lemma}\label{l:efnBound}
  Let $\eta > 0$ be as in Lemma~\ref{l:TstarLower}, and define
  \begin{equation}\label{e:Ndef}
    N \defeq \ceil[\bigg]{ \frac{d \ln (\epsilon \eta)}{\ln (p_{\max}) } }\,.
  \end{equation}
  There is $\chi > 0$ such that for all sufficiently small~$\epsilon > 0$ and every probability measure~$\mu$, we have
  \begin{equation*}
    \frac{d (T_*^{N+2} \mu)}{d\pi} \geq \chi \,.
  \end{equation*}
\end{lemma}

The expression $d (T_*^{N+2} \mu) / d\pi$ denotes the Radon--Nikodym derivative of~$T_*^{N+2} \mu$ with respect to the Lebesgue measure~$\pi$.
Momentarily postponing the proofs of Lemmas~\ref{l:FsExist}--\ref{l:efnBound}, we prove Theorem~\ref{t:TMixUpper}.
\begin{proof}[Proof of Theorem~\ref{t:TMixUpper}]
  Let~$\mu_0$ be any probability measure, and inductively define
  \begin{equation*}
    \mu_{n+1}  = \frac{T_*^{N+2} \mu_n - \chi \pi}{1 - \chi}\,.
  \end{equation*}
  By Lemma~\ref{l:efnBound} we see that~$\mu_n$ is a positive measure, and hence by our normalization must be a probability measure.
  Since
  \begin{equation*}
    T_*^{n(N+2)}\mu_0 = (1 - (1-\chi)^n) \pi + (1 - \chi)^n \mu_n\,,
  \end{equation*}
  we note
  \begin{equation*}
    \norm{T_*^{n(N+2)} \mu_0 - \pi}_\TV \leq (1 - \chi)^n\,,
  \end{equation*}
  which immediately implies
  \begin{equation*}
    \tmix \leq \frac{(N+2) \ln 2}{\abs{\ln(1 - \chi)} }\,.
  \end{equation*}
  Using the definition of~$N$ (equation~\eqref{e:Ndef}) we obtain~\eqref{e:tMixGaussian} as claimed.
  \medskip
\end{proof}

\subsection{A lower bound for \texorpdfstring{$T_*^n \bm F_s$}{T*n Fs} (Lemma~\ref{l:TstarLower}).}

The main idea behind the proof of Lemma~\ref{l:TstarLower} is to control the mass that escapes the envelope of $\bm F_s$ through the noise.
Once this is established, repeated application of $U_*$ to~$\bm F_s$ will give a function that bounded away from~$0$.

\begin{proof}[Proof of Lemma~\ref{l:TstarLower}]
  Using~\eqref{e:UStarFs} and~\eqref{e:KepsFs} we see
  \begin{equation}\label{e:TStarFs1}
    T_* \bm F_s
    = K_\epsilon * U_* \bm F_s
    \geq (1 - a (\lambda_{\sigma s} \epsilon)^\gamma ) \bm F_{\sigma s}\,.
  \end{equation}
  Since $\ell_s \geq \epsilon \eta = \epsilon (2a)^{1/\gamma}$ by assumption, we must have $(1 - a (\lambda_s \epsilon)^\gamma ) \geq 1/2$. Therefore,
  \begin{equation}\label{e:expLower}
    1 - a (\lambda_{\sigma s} \epsilon)^\gamma  \geq e^{-C (\lambda_{\sigma s} \epsilon)^\gamma}
  \end{equation}
  where $C = 2 \ln 2$ is independent of $\epsilon$.
  Using this and iterating~\eqref{e:TStarFs1} gives
  \begin{equation}\label{e:TstarS}
    T_*^{n} \bm F_s
      \geq \paren[\Big]{
	  \prod_{k = 1}^{n} (1 -  a (\lambda_{\sigma^k s} \epsilon)^\gamma )
	} \bm F_{\sigma^n s}
      \geq
	\exp\paren[\Big]{ -C  \epsilon^\gamma \sum_{k=1}^{n} \lambda_{\sigma^k s}^\gamma  }
	\bm F_{\sigma^n s}
	\,.
  \end{equation}
  for any $n \in \N$.

  The sum in the exponential is easily bounded for $n \leq \abs{s}$.
  Indeed, if  $s = (s_0, \dots, s_{n'-1}) \in \mathcal T $, then 
  \begin{equation*}
    \lambda_{\sigma s}
      = \frac{1}{\pi(\cyl_{\sigma s})^{1/d}}
      = \frac{p_{s_0}^{1/d}}{\pi(\cyl_{s})^{1/d}}
      \leq p_{\max}^{1/d} \lambda_s\,.
  \end{equation*}
  Hence for every $n \leq \abs{s}$ we have
  \begin{equation}\label{e:tmp-2023-01-30-1}
    \epsilon^\gamma \sum_{k = 1}^{n} \lambda_{\sigma^k s}^\gamma
      \leq \frac{(\epsilon \lambda_{\sigma s})^\gamma}{1 - p_{\max}^{\gamma/d} }
      \leq  \frac{1}{\eta^\gamma (1 - p_{\max}^{\gamma/d}) }\,,
  \end{equation}
  which is finite and independent of $n$ and $\epsilon$.
  Using this in~\eqref{e:TstarS} implies~\eqref{e:TStarLower} holds, for all $n \leq \abs{s}$ and some nonnegative, $\epsilon$-independent constant~$\beta > 0$ that can be computed explicitly to be
  \[
  \beta = \exp \paren[\Big]{\frac{- 2 \ln (2) }{2 a( 1 - p_{\max}^{\gamma/d}) }} \,.
  \]

  To handle the case $n > \abs{s}$ we note first that we have already proved
  \begin{equation}\label{e:TstarS1}
    T_*^{\abs{s}} \bm F_s \geq \beta \bm F_\zerotuple\,.
  \end{equation}
  By assumption~$\check{\bm F}_\zerotuple$ is continuous and strictly positive in~$(0, 1)^d$, and hence condition~\ref{i:sep} in Assumption~\ref{a:Cubes} implies $U_* \bm F_\zerotuple > 0$ on $\T^d$.
  This in turn implies $T_* \bm F_\zerotuple > 0$ on~$\T^d$.
  Since $\bm F_\zerotuple$ is continuous, this implies $\min_{\T^d} \bm F_\zerotuple > 0$, and hence applying~$T_*$ to~\eqref{e:TstarS1} shows
  \begin{equation*}
    T_*^{\abs{s}+1} \bm F_s \geq \beta \paren[\Big]{ \min_{\T^d} T_* \bm F_\zerotuple } = \beta'\,.
  \end{equation*}
  This implies~\eqref{e:TStarLower} for $n = \abs{s} + 1$, and some finite constant~$\beta'$ that is independent of~$s, \epsilon$.
  Since~$T_*$ can not decrease a nonnegative minimum, we also obtain~\eqref{e:TStarLower} for all $n \geq \abs{s} + 1$, concluding the proof.
\end{proof}

\subsection{A lower bound on \texorpdfstring{$T_*^n \mu$}{T*n mu} (Lemma~\ref{l:efnBound})}
\label{s:efnBoundProof}
The main idea behind the proof of Lemma~\ref{l:efnBound} is to partition the torus into cylinder sets $\cyl_s$ with side length~$\ell_s = O(\epsilon)$.
If we apply~$T_*$ once to the initial measure~$\mu$, then $X_1$ has a density that is roughly uniform on sets at the scale $\epsilon$, and hence can be bounded from below by linear combination of functions~$\bm F_s$ with non-negative coefficients.
Applying Lemma~\ref{l:TstarLower} to this will allow us to show that the distribution eventually becomes bounded away from~$0$.
\begin{proof}[Proof of Lemma~\ref{l:efnBound}]
We first define~$\mathcal{S}_\epsilon$ by 
  \begin{equation*}
    \mathcal S_{\epsilon} = \set[\big]{ s \in \mathcal T \st
	  ~\ell_{\sigma s} > \eta  \epsilon, \quad
	\ell_{s} \leq \eta  \epsilon },
  \end{equation*}
where we recall that $a$ is the constant appearing in~Lemma~\ref{l:FsExist} (see the rightmost figure in Figure~\ref{f:CylSubdivisionFigure} for an illustration).
For all $s \in \mathcal{T}$ with $|s| \geq 1$, the side length~$\ell_s$ satisfies 
\[
\left(p_{\min}^{1/d}\right)^{|s|} \leq \ell_s \leq \left(p_{\max}^{1/d}\right)^{|s|}, \quad \quad p_{\min}^{1/d} \ell_{\sigma s}  \leq \ell_{s} \leq p_{\max}^{1/d} \ell_{\sigma s}. 
\]
So, for every $s \in \mathcal{S}_\epsilon$ we must have
\begin{equation}\label{e:diamCSN}
 p_{\min}^{1/d} \eta  \epsilon \leq  \ell_{s} \leq \eta  \epsilon,
\end{equation}
and
\begin{equation}
|s| - 1 \leq N = \left \lfloor \frac{d \ln (\epsilon \eta)}{\ln p_{\max}} \right \rfloor. \label{absslower}
\end{equation}

  Since $\check K \in L^1$, the measure $T_* \mu$ is absolutely continuous with respect to the Lebesgue measure~$\pi$, and we let
  \begin{equation*}
    f \defeq \frac{d (T_*\mu)}{d\pi} \,,
  \end{equation*}
  denote the Radon--Nikodym derivative of~$T_* \mu$.
  Define
  \begin{equation}\label{e:csDef}
    c_s \defeq \frac{1}{\norm{\bm F_s}_{L^\infty(\cyl_s)}} \inf_{x \in \cyl_s} f(x) \,, \quad \quad  s \in \mathcal{S}_\epsilon
  \end{equation}
  and note each $c_s$ is nonnegative and
  \begin{equation}\label{e:fLower2}
    f(x) \geq \sum_{s \in \mathcal S_\epsilon} c_s \bm F_s(x) \,.
  \end{equation}
  We now claim that
  \begin{equation}\label{e:fEfnLower}
    \sum_{s \in \mathcal S_\epsilon} c_s \geq \frac{1}{C_0} \,,
  \end{equation}
  for some constant~$C_0 > 0$ that is independent of~$\epsilon$.

  To prove~\eqref{e:fEfnLower}, let $\nu = U_*\mu$.
  Note that the lower bound in~\eqref{e:diamCSN} implies that for any $s \in \mathcal S_\epsilon$, $\cyl_s$ is a cube with side length at most $\eta \epsilon$.
  Thus, for any $s \in \mathcal S_\epsilon$ and $x \in \cyl_s$, we have
  \begin{align}\label{e:fLower1}
    f(x)  = K_\epsilon * \nu (x)  \geq \int_{\cyl_s} K_\epsilon (x - y) \, \nu(dy)  \geq \frac{\nu(\cyl_s)}{\epsilon^d} \kappa
  \end{align}
where
\[
\kappa = \inf_{\check x, \check y \in [0, \eta)^d} \check K(\check x - \check y) > 0.
\]
By~\eqref{e:UStarFs} and~\eqref{e:diamCSN} we note
  \begin{equation}\label{e:FsLower1}
    \norm{\bm F_s}_{L^\infty}
      = \frac{1}{\pi(\cyl_s)} \norm{\bm F_\zerotuple}_{L^\infty}
      \leq \frac{C_1}{\epsilon^d}\,,
  \end{equation}
  for some constant~$C_1$ that is independent of~$\epsilon$.
  Using~\eqref{e:FsLower1} and~\eqref{e:fLower1} in \eqref{e:csDef} we obtain
  \begin{equation*}
    c_s \geq \frac{\nu(\cyl_s) \kappa}{C_1},\, \quad \forall \; s \in \mathcal{S}_\epsilon.
  \end{equation*}
  Since the sets $\{ \cyl_s \;|\; s \in \mathcal S_\epsilon\}$ form a partition and~$\nu$ is a probability measure, summing the above over $s \in \mathcal S_\epsilon$ yields~\eqref{e:fEfnLower} as desired, with $C_0 = C_1/\kappa$.

  Now, to finish we note that~$T_*$ is monotone on nonnegative functions.  That is, if $g_1$, $g_2$ are any two functions such that $0 \leq g_1 \leq g_1$, then $0 \leq T_* g_1 \leq T_* g_2$.  By \eqref{absslower}, $|s| \leq 1 + N$ for all $s \in \mathcal{S}_\epsilon$. Therefore, using~\eqref{e:diamCSN}, \eqref{e:fLower2} and Lemma~\ref{l:TstarLower} implies
  \begin{align*}
    T_*^{N+1} f &\geq \sum_{s \in \mathcal S_\epsilon} c_s T_*^{N+1} \bm F_s
      \geq \sum_{s \in \mathcal S_\epsilon} c_s  \beta' \geq \frac{\beta'}{C_0}.
  \end{align*}
where $C_0 > 0$ is the constant from~\eqref{e:fEfnLower}, independent of $\epsilon$.
Choosing $\chi = \beta'/C_0$, the proof is complete.
\end{proof}

\subsection{Constructing \texorpdfstring{${\bm F}_s$}{Fs} (Lemma \ref{l:FsExist}).}\label{s:FsExist}

We now construct the family of functions~$\bm F_s$.
In light of the property~\eqref{e:UStarFs}, it is enough to find one function~$\check{\bm F}_\zerotuple$ that satisfies a bound like~\eqref{e:KepsFs}.

\begin{lemma}\label{l:F0Exist}
  Suppose either Assumption~\ref{a:Gaussian} or~\ref{a:Separated} holds.
  There exists a continuous function~$\check{\bm F}_\zerotuple \colon \R^d \to [0, \infty)$, and constants $a < \infty$, $\gamma > 0$ such that
  \begin{gather}
    \set{ \check{\bm F}_\zerotuple > 0 } = (0, 1)^d\,,
    \qquad\int_{\R^d} \check{\bm F}_\zerotuple \, d\check \pi = 1\,,
    \\
    \label{e:KepsFzt}
    \text{and}\quad \check K_\epsilon * \check{\bm F}_\zerotuple \geq (1 - a \epsilon^\gamma) \check{\bm F}_\zerotuple\,,
  \end{gather}
  for all~$\epsilon > 0$.
\end{lemma}

Given the function~$\check{\bm F}_\zerotuple$, we construct the functions~$\bm F_s$ by rescaling, and projecting to the torus.

\begin{proof}[Proof of Lemma~\ref{l:FsExist}]
  For any~$s \in \mathcal T$ we define
  \begin{equation*}
    \check{\bm F}_s
    = \frac{1}{\pi(\cyl_s)} \one_{\check \cyl_s} \check{\bm F}_\zerotuple\circ \check\varphi^{\abs{s}} \,,
  \end{equation*}
  where
  \begin{equation*}
    \check \cyl_s = \set{ \check x \in [0, 1)^d \st \Pi(\check x) \in \cyl_s }\,,
    \quad\text{and}\quad
    \check\varphi(\check x) = \check\varphi_i(\check x) \quad\text{if } \check x \in \check E_i\,.
  \end{equation*}
  We project these functions to the torus by defining
  \begin{equation*}
    \bm F_s(x) \defeq \check{\bm F}_s(\check x)\,,
  \end{equation*}
  where as before $\check x \in [0, 1)^d$ is the unique point such that $\Pi(\check x) = x$.
  We claim the functions~$\bm F_s$ satisfy all the properties required in the statement of Lemma~\ref{l:FsExist}.

  Clearly~$\bm F_s$ is supported on~$\cyl_s$, and is strictly positive in the interior.
  By definition of~$U_*$ we see~\eqref{e:UStarFs} holds, which in turn implies~$\norm{\bm F_s}_{L^1} = \norm{\check{\bm F}_\zerotuple}_{L^1} = 1$.
  It only remains to verify~\eqref{e:KepsFs}.
  To see this, note first that for all $x \in \cyl_s$ we have
  \begin{equation*}
    \check \varphi^{\abs{s}} (x) = \lambda_s \check D \check x + \check e_s
  \end{equation*}
  for some $\check e_s \in \R^d$, and orthogonal matrix~$D$ that is a product of the matrices~$\check D_i$ in~\eqref{e:PhiDef}.
  Thus, setting $n = \abs{s}$, we see
  \begin{align*}
    \check K_\epsilon * \check{\bm F}_{s}(\check x)
      &= \frac{1}{\pi(\cyl_s) \epsilon^d} \int_{\R^d} \check K\paren[\Big]{ \frac{\check y}{\epsilon} } \check{\bm F}_\zerotuple \paren[\Big]{ \lambda_s \check D (\check x - \check y) + \check e_s } \, d\check y
      \\
      &= \frac{1}{\epsilon^d} \int_{\R^d} \check K\paren[\Big]{ \frac{\check D^{-1} \check y}{\lambda_s \epsilon} } \check{\bm F}_\zerotuple ( \check \varphi^n ( \check x ) - \check y ) \, d \check y
      = \lambda_s^d (\check K_{\lambda_s \epsilon } * \check{\bm F}_\zerotuple)(\check \varphi^n (\check x))\,,
  \end{align*}
  where the last equality followed because of~\eqref{e:KSymm}.
  Using~\eqref{e:KepsFzt} (with~$\epsilon$ replaced by $\lambda_s \epsilon$), we note
  \[
    (\check K_{\lambda_s \epsilon } * \check{\bm F}_\zerotuple) \geq\paren{ 1 - a (\lambda_s \epsilon)^\gamma }
    \check{\bm F}_\zerotuple,
  \]
  and hence
  \begin{align*}
    \check K_\epsilon * \check{\bm F}_{s}(\check x)  &\geq \frac{1}{\pi(\cyl_s)} \paren{ 1 - a (\lambda_s \epsilon)^\gamma }
    \check{\bm F}_\zerotuple\circ \check \varphi^n (\check x)
    = \paren{ 1 - a (\lambda_s \epsilon)^\gamma }
    \check{\bm F}_s(\check x)\,.  
  \end{align*}
  This implies
  \begin{equation*}
    K_\epsilon * \bm F_s(x)
    \geq \check K_\epsilon * \check{\bm F}_s(\check x)
    \geq (1 - a (\lambda_s \epsilon)^\gamma ) \check{\bm F}_s( \check x)
    = (1 - a (\lambda_s \epsilon)^\gamma ) \bm F_s(x) \,,
  \end{equation*}
  yielding~\eqref{e:KepsFs} as claimed.  
\end{proof}

It remains to prove Lemma~\ref{l:F0Exist}.
Under Assumption~\ref{a:Gaussian}, we can just choose~$\bm F_s$ to be the principal eigenfunctions of the associated elliptic operator, and we will obtain~\eqref{e:KepsFs} with~$\gamma = 2$.
We do this next.

\begin{proof}[Proof of Lemma~\ref{l:FsExist} under Assumption~\ref{a:Gaussian}]
  Let $\check A = (\check a_{i,j})$ be the covariance matrix of~$\check K$, and define the differential operator~$\mathcal L$ by
  \begin{equation*}
    \mathcal L = \frac{1}{2} \sum_{i,j = 1}^d \check a_{i,j} \partial_i \partial_j\,.
  \end{equation*}
  Since~$\check A$ is non-degenerate, by assumption, the operator~$\mathcal L$ is elliptic.
  Let~$\check{\bm F}_\zerotuple$ be the principal eigenfunction of~$\mathcal L$ on the unit square $[0, 1]^d$, with Dirichlet boundary conditions, normalized so that $\norm{\check{\bm F}_\zerotuple}_{L^1} = 1$.
  Since~$\mathcal L$ is elliptic we know that $\check{\bm F}_\zerotuple$ can be chosen to be strictly positive on $(0, 1)^d$, and the associated principal eigenvalue~$\check \lambda_0 > 0$.

  Define the functions~$\theta, \underline{\theta}$ by
  \begin{equation*}
    \theta(t,x) = \check K_{\epsilon \sqrt{t}} * \check{\bm F}_\zerotuple(x)\,,
    \quad\text{and}\quad
    \underline{\theta}(t,x)
      = e^{ -\check \lambda_0 \epsilon^2 t } \check{\bm F}_0(x)\,.
  \end{equation*}
  Since $\check K_\epsilon$ is a Gaussian we know the function~$\theta$ satisfies the diffusion equation
  \begin{equation}\label{e:heat}
    \partial_t \theta = \epsilon^2 \mathcal L \theta \quad\text{in } \R^d\,,
  \end{equation}
  with initial data~$\theta(0,x) = \check{\bm F}_\zerotuple$ (extended by $0$ outside the cube $[0,1]^d$).
  Since
  \begin{equation*}
    - \mathcal L \check{\bm F}_\zerotuple = \check \lambda_0 \check{\bm F}_\zerotuple \quad\text{in } (0, 1)^d\,,
  \end{equation*}
  and the outward normal derivative of $\check{\bm F}_\zerotuple$ is nonpositive on the boundary of the cube $[0, 1]^d$, the function~$\underline{\theta}$ is a sub-solution to the diffusion equation~\eqref{e:heat}.
  As a result we must have $\underline\theta_t \leq \theta_t$ for all $t \geq 0$ and $x \in [0,1]^d$.
  Setting~$t = 1$ yields
  \begin{equation*}%
    \check K_\epsilon * \check{\bm F}_\zerotuple \geq
      e^{-\check \lambda_0 \epsilon^2 } \check{\bm F}_\zerotuple\,.
  \end{equation*}
  Since $e^{-t} \geq 1 - t$, we also obtain~\eqref{e:KepsFzt} with $\gamma = 2$ and $a = \check \lambda_0$, concluding the proof.
\end{proof}
\begin{remark*}
  When~$\check K$ is the standard Gaussian, then
  \begin{equation*}
    \check{\bm F}_\zerotuple(\check x)
      = \prod_{k=1}^d \sin( \pi \check x_k ) \,,
    \quad\text{and}\quad
    \lambda_0 = \frac{\pi^2 d}{2}\,.
  \end{equation*}
\end{remark*}

It remains to prove Lemma~\ref{l:F0Exist} under Assumption~\ref{a:Separated}.
This is more involved, and we present the proof in Section~\ref{s:F0Exist}.

\subsection{Constructing \texorpdfstring{$\check{\bm F}_\zerotuple$}{F0} under Assumption~\ref{a:Separated} (Lemma \ref{l:F0Exist}).}\label{s:F0Exist}

In the non-Gaussian case (Assumption~\ref{a:Separated}), we will start with~$d = 1$ and choose~$\check{\bm F}_\zerotuple$ to be a tent function.
This will eventually yield~\eqref{e:KepsFzt} with~$\gamma = 1$.
We begin with the calculation in one dimension.

\begin{lemma}\label{l:tent1d}
  Let $\check f \colon \R \to \R$ be the tent like function defined by
  \begin{equation*}
    \check f(\check x) =
      \begin{dcases}
	\check x & 0 \leq \check x < \frac{1}{2}\\
	1 - \check x & \frac{1}{2} \leq \check x \leq 1\,,\\
	0 & \text{otherwise}\,.
      \end{dcases}
  \end{equation*}
  If $\check K:\R \to \R$ is an even function, then $\check{\bm F}_\zerotuple = \check f / \norm{f}_{L^1}$ satisfies~\eqref{e:KepsFzt} with~$\gamma = 1$, and
  \begin{equation*}
    a = 4 \int_{-\infty}^\infty \abs{\check y} \check K(\check y) \, d\check y\,.
  \end{equation*}
\end{lemma}
\begin{proof}
Because $\check K$ is even, and $f$ is symmetric about $\check x = 1/2$, it suffices to prove the bound for $\check x \in (0,1/2]$.  For $0 < \check x \leq 1/4$ we note
  \begin{align*}
    \MoveEqLeft
    \check f(\check x) - \check K_\epsilon*\check f ( \check x)
    = \int_{-\infty}^\infty \paren{ \check x - \check f(\check x-\check y) } \check K_\epsilon(\check y) \, d\check y
  \\
    &\leq \int_{\check x - \frac{1}{2}}^{\check x} \check y \check K_\epsilon(\check y) \, d\check y
      + \check x \int_{\check y \notin (\check x - \frac{1}{2}, \check x)} \check K_\epsilon(\check y) \, d\check y
    \\
    &\leq \int_{-\check x}^{\check x} \check y \check K_\epsilon(\check y) \, d\check y
      + \check x \int_{\abs{\check y} \geq \frac{1}{4} } \check K_\epsilon(\check y) \, d\check y
    \leq 4\epsilon f(\check x) \int_{-\infty}^\infty \abs{\check y} \check K(\check y) \, d\check y \,,
  \end{align*}
  since~$\check K$ is even.  For $1/4 < \check x \leq 1/2$ we note $f(\check x) \geq 1/4$ and so
  \begin{align*}
    \check f(\check x) - \check K_\epsilon*\check f ( \check x)
    &\leq \int_{-\infty}^\infty \abs{\check x - \check f(\check x-\check y)} \, \check K_\epsilon(\check y) \, d\check y
    \leq \int_{-\infty}^\infty \abs{\check y} \check K_\epsilon(\check y) \, d\check y
    \\
    &\leq 4\epsilon f(\check x) \int_{-\infty}^\infty \abs{\check y} \check K(\check y) \, d\check y\,.
  \end{align*}
Thus whenever $\check f > 0$ we have the estimate
  \begin{equation*}
    \check f - \check K_\epsilon*\check f \leq a \epsilon \check f\,,
  \end{equation*}
  which immediately implies~\eqref{e:KepsFzt} with~$\gamma = 1$.
\end{proof}

Next, in arbitrary dimension $d \geq 1$, we construct $\check{\bm F}_\zerotuple$ in the case that~$\check K$ is separated.
\begin{lemma}\label{l:sepK}
  Suppose~$\check K:\R^d \to \R$ is of the form
  \begin{equation}\label{e:Ksep2}
    \check K(\check x) = \prod_{i=1}^d \check k_i( \check x_i )\,,
  \end{equation}
  for some even functions $\check k_i:\R \to \R$ such that
  \begin{equation}\label{e:Ksep2Moments}
    \int_{\R} \check k_i(\check x_i) \, d\check x_i = 1\,,
    \quad
    \int_{\R} \abs{\check x_i} \check k_i(\check x_i) \, d\check x_i  = A \,,
    \quad
    \int_0^{\frac{1}{2}}
      \check x_i \check k_i(\check x_i) \, d\check x_i  = \underline{A} \,,
  \end{equation}
  for constants~$A < \infty$ and $\underline A > 0$.
  Let $\check f$ be the tent function from Lemma~\ref{l:tent1d}, and define
  \begin{equation}\label{e:F0Def}
    \check F(\check x) = \prod_{i=1}^d \check f(\check x_i)\,,
    \quad\text{and}\quad
    \check{\bm F}_\zerotuple = \frac{\check F}{\norm{\check F}_{L^1}}\,.
  \end{equation}
  There exists a constant~$a = a(d, A, \underline A)$ such that~\eqref{e:KepsFzt} holds with~$\gamma = 1$ and all sufficiently small~$\epsilon > 0$.
\end{lemma}

\begin{proof}
  Note first for any $\check x \in (0, 1)^d$,
  \begin{equation*}
    \check K_\epsilon * \check F(x) \geq \check K_\epsilon * \check F(0)
      \geq \prod_{i=1}^d \int_0^{1/2} \check x_i \check k_{i,\epsilon} (\check x_i)  \, d\check x_i
      \geq \paren[\Big]{\frac{\underline A \epsilon}{2}}^d\,,
  \end{equation*}
  for all sufficiently small~$\epsilon$.
  Thus whenever $\check F(\check x) < (\underline{A} \epsilon/2)^d$ we have
  \begin{equation}\label{e:CkF-KepF1}
    \check F(\check x) - \check K_\epsilon * \check F(\check x) \leq 0 \,.
  \end{equation}

  Now suppose $\check F(\check x) \geq (\underline A \epsilon/2)^d$.
  Then for at least one $i \in \set{1, \dots, d}$ we must have $\check f(\check x_i) \geq \underline A \epsilon/ 2$.
  For simplicity, and without loss of generality, we assume $i = 1$.
  We will now use the notation $\check x = (\check x_1, \check x')$ where $\check x' = (\check x_2, \dots, \check x_d)$, and $\check F'(\check x') = \prod_2^d \check f(\check x_i)$, etc.
  By induction on~$d$ we may also assume
  \begin{equation*}
    \check F' - \check K_\epsilon' * \check F'
      \leq a \epsilon \check F'\,,
  \end{equation*}
  for some dimensional constant~$a = a(A, \underline A)$.
  We will subsequently allow~$a$ to increase from line to line, as long as it does not depend on~$\epsilon$ or $\check x$. 

  Now, we compute
  \begin{align}
    \nonumber
    \MoveEqLeft
    \check F(\check x) - \check K_\epsilon * \check F(\check x)
      = \int_{\R^d} ( \check f(\check x_1) \check F'(\check x') - \check f(\check x_1 - \check y_1) \check F'( \check x' - \check y') ) \check K_\epsilon(\check y) \, d\check y
    \\
    \nonumber
      &= \int_{\R^d} 
	  \check F'(\check x') ( \check f( \check x_1)  - \check f(\check x_1 - \check y_1) ) 
	  \check K_\epsilon(\check y) \, d\check y
    \\
    \nonumber
	&\qquad+  \int_{\R^d}  \check f(\check x_1 - \check y_1) ( \check F'(\check x') - \check F'(\check x' - \check y') ) \check K_\epsilon(\check y) \, d\check y
    \\
    \label{e:CkF-KepF2}
      &\leq a \epsilon \check F'(\check x') (\check f(\check x_1) +  \check k_\epsilon * \check f(\check x_1) ) \,.
\end{align}
Above we used Lemma~\ref{l:tent1d} to bound the first term and the induction hypothesis to bound the second integral.

  Observe that
  \begin{equation*}
    \abs[\big]{ \check f(x_1) - \check k_\epsilon * \check f(\check x_1)  }
      \leq \epsilon \norm{\grad \check f}_{L^\infty} \int_\R \abs{\check x_1} \check k_1(\check x_1) \, d\check x_1\,.
  \end{equation*}
  Using this and the assumption $\check f(\check x_1) \geq \underline A \epsilon / 2$ in~\eqref{e:CkF-KepF2} yields
  \begin{equation*}
    \check F(\check x) - \check K_\epsilon * \check F(\check x)
      \leq a \epsilon \check F'(\check x') ( \check f(\check x_1) + \epsilon )
      \leq a \epsilon \check F'(\check x') \paren[\Big]{ \check f(\check x_1) + \frac{2 \check f(\check x_1)}{\underline A} }
      \leq a \epsilon \check F(\check x) \,.
  \end{equation*}
  Combining this with~\eqref{e:CkF-KepF1} concludes the proof.
\end{proof}

Given Lemma~\ref{l:sepK}, a standard approximation argument can be used to deduce Lemma~\ref{l:F0Exist}.
\begin{proof}[Proof of Lemma~\ref{l:F0Exist} under Assumption~\ref{a:Separated}]
  Let~$\check K_n$ be as in~\eqref{e:Ksep3}, and define
  \begin{equation*}
    \check K_{n,i}'(x) = \prod_{j=1}^d \check k_{n,i}( \check x_j )\,.
  \end{equation*}
  Let~$\check{\bm F}_\zerotuple$ be the function defined in~\eqref{e:F0Def}.
  By Lemma~\ref{l:sepK} we know there exists $a = a(A, 1/A, d)$, independent of~$n, i$ such that 
  \begin{equation*}
    \check K_{n, i}' * \check{\bm F}_\zerotuple \geq (1 - a \epsilon ) \check{\bm F}_\zerotuple\,.
  \end{equation*}
  Multiplying by~$b_{n, i}$ (which are nonnegative by assumption) and summing yields
  \begin{equation*}
    \check K_{n}' * \check{\bm F}_\zerotuple
      \geq  (1 - a \epsilon ) \check{\bm F}_\zerotuple \sum_{i=1}^n b_{n, i}
      =  (1 - a \epsilon ) \check{\bm F}_\zerotuple \,.
  \end{equation*}
  Since~$\check K_n \to \check K$ in~$L^1$ taking the limit as~$n \to \infty$ yields~\eqref{e:KepsFzt} as desired.
\end{proof}

\section{Lower Bounds on the Mixing Time.}\label{s:MixTimeLower}
\subsection{Proof of the lower bound (Theorem \ref{t:TMixLower}).}
The main idea behind the proof of Theorem~\ref{t:TMixLower} is to choose~$X_0$ to be uniformly distributed on a cylinder set, and track the distribution of~$X_n$ carefully.
The action of~$\varphi$ pushes the distribution to be uniform on a cylinder set that is one order lower.
If the noise does not change this too much, then the mixing time can be bounded below by the order of the cylinder set.

To make this idea rigorous, we need to estimate $\norm{T_*^n f - U_*^n f}_{L^1}$ for functions which are piecewise constant on cylinder sets (recall~$T_*$ and~$U_*$ are defined in~\eqref{e:TStarDef} and~\eqref{e:DefUTStar} respectively).
For notational convenience given~$s \in \mathcal T$ we define the normalized indicator function $\bm I_s$ by
\begin{equation*}%
  \bm I_s \defeq \frac{1}{\pi(\cyl_s)} \one_{\cyl_s}\,.
\end{equation*}
The action of~$\varphi$ on such functions is explicit.
Indeed, for any $s = (s_0, s_1, \dots, s_n) \in \mathcal S$ we see
\begin{equation*}%
  U_* \one_{\cyl_s}
    = \sum_{i=1}^M \one_{\cyl_s} \circ \varphi_i^{-1} p_i
    = \one_{\cyl_{s}} \circ \varphi_{s_0}^{-1} p_{s_0}
    = p_{s_0} \one_{\cyl_{\sigma s}}\,,
\end{equation*}
and hence
\begin{equation}\label{e:evolutionIdentity}
  U_* \bm I_{s}
    = \frac{1}{\pi(\cyl_s)} U_* \one_{\cyl_s}
    = \frac{p_{s_0}}{\pi(\cyl_s)} \one_{\cyl_{\sigma s}}
    = \bm I_{\sigma s}\,.
\end{equation}
Thus if $n < \abs{s}$, then $U_*^n \bm I_s = \bm  I_{\sigma^n s}$, which is not mixed and can be used to give a lower bound on the mixing time.

In order to bound~$\norm{T_*^n f - U_*^n f}_{L^1}$ we need to control the amount of mass that leaks out of cylinder sets due to the action of the noise.
We will shortly see that this is bounded by the \emph{perimeter to volume ratio}, which we denote by~$H$.
Explicitly, if~$\mathcal S \subseteq \mathcal T$, we define
\begin{equation}\label{e:HDef}
  H(\mathcal S) \defeq
    \max_{s \in \mathcal S \setminus \{\zerotuple\}} \frac{\Per(\cyl_s)}{\pi(\cyl_s)}\,,
  \quad\text{and by convention}\quad H(\zerotuple) = 0 \,.
\end{equation}

In our case, all the cylinder sets are cubes by assumption, in which case~$H(\mathcal S)$ can be expressed in terms of~$\mathcal \lambda_s$ (which we recall is defined in~\eqref{elldef}).
Indeed, if~$\mathcal S \neq \set{\zerotuple}$, then
\begin{equation}\label{e:HLowerBound}
  H(\mathcal S)
    = \max_{s \in \mathcal S \setminus \set{\zerotuple}} \frac{2d}{\ell_s}
    = \max_{s \in \mathcal S \setminus \set{\zerotuple}} 2d \lambda_s
    \,.
\end{equation}

We now present a lemma controlling the error~$\norm{T_*^n f - U_*^n f}_{L^p}$ for piecewise constant functions and any $p \in [1, \infty)$.
To prove Theorem~\ref{t:TMixLower} we only need~$p = 1$.
However, in order to prove Theorem~\ref{t:pl-dis-time} we will need~$p = 2$.

\begin{lemma}\label{l:pcMix}
  Suppose that either $p = 1$ and~$K_\epsilon$ satisfies~\eqref{e:FirstMomentKep},  or~$p > 1$ and~$K_\epsilon$ is supported in the ball $B(0, \epsilon)$.
  Let~$\mathcal S \subseteq \mathcal T$ be a finite set so that~$\set{\cyl_s \st s \in \mathcal S}$ partitions the torus~$\T^d$, and
  \begin{equation}\label{e:Leps}
    p_{\min}^{1/d} L \epsilon \leq \ell_s \leq L \epsilon \,,
  \end{equation}
  for some constant~$L\geq 1$.
  Suppose~$f \in L^p$ is of the form
  \begin{equation}\label{e:f01}
    f_0 = \sum_{s \in \mathcal S} c_0(s) \one_{\cyl_s} \,.
  \end{equation}
  Then for all~$N \in \N$ we have
  \begin{equation}\label{e:fullCylLpEstimate}
    \norm{ T_*^N f_0 - U_\ast^N f_0 }_{L^p}
      \leq \epsilon^{1/p} C_1^{1/p'} 
	\sum_{n = 1}^{N} H( \sigma^n \mathcal S )^{1/p} \norm{f_0}_{L^p}\,,
  \end{equation}
  where~$p' = p / (p-1)$ is the H\"older conjugate of~$p$, and
  \begin{equation}\label{e:C1Def}
    C_1 \defeq 2d \paren[\Big]{ 2 + \frac{1}{p_{\min}^{1/d}} }^{d -1}\,.
  \end{equation}
\end{lemma}

In our case the sum on the right of~\eqref{e:fullCylLpEstimate} can be bounded explicitly.
Indeed, if $s = (s_0, \dots, s_n) \in \mathcal T$, we note
\begin{equation*}
  \ell_{\sigma s} = \pi(\cyl_{\sigma s})^{1/d}
    = \paren[\Big]{ \frac{\pi(\cyl_s)}{p_{s_0}} }^{1/d}
    \geq \frac{\ell_s}{p_{\max}^{1/d}}\,,
\end{equation*}
and hence $H(\sigma \mathcal S) \leq p_{\max}^{1/d} H(\mathcal S)$.
This immediately implies
\begin{equation}\label{e:HSummable}
  \sum_{n=1}^{N} H( \sigma^n (\mathcal S) )^{1/p}
    \leq \frac{H(\mathcal S)^{1/p}}{1 - p_{\max}^{1/(pd)}} \,.
\end{equation}
Thus, we obtain the following corollary.

\begin{corollary}\label{c:TStarUstar}
  Suppose either $p = 1$ and~$K_\epsilon$ satisfies~\eqref{e:FirstMomentKep},  or~$p > 1$ and~$K_\epsilon$ is supported in $B(0, \epsilon)$.
  Let~$\delta > 0$, choose
  \begin{equation}\label{e:LambdaPDelta}
    \Lambda_{p, \delta} \defeq \frac{2 d C_1^{p-1}}{\delta^p p_{\min}^{1/d} (1 - p_{\max}^{1/pd})^p}\,,
  \end{equation}
  and define $\mathcal S = \mathcal S_{\epsilon, \delta} \subseteq \mathcal T$ by
  \begin{equation}\label{e:SepDelta}
    \mathcal S_{\epsilon, \delta} = \set{
      s \in \mathcal T \st \ell_s \leq \epsilon \Lambda_{p, \delta}\,,
	\text{ and } \ell_{\sigma s} >  \epsilon \Lambda_{p, \delta} }\,.
  \end{equation}
  If~$f_0$ is defined by~\eqref{e:f01}, then for every~$N \in \N$ we have
  \begin{equation}\label{e:fullCylLpEstimate2}
    \norm{ T_*^N f_0 - U_\ast^N f_0 }_{L^p}
      \leq \delta \norm{f_0}_{L^p} \,.
  \end{equation}
\end{corollary}
\begin{proof}
  Note~$\set{\cyl_s \st s \in \mathcal S_{\epsilon, \delta}}$  partitions the torus, and for every~$s \in \mathcal S_{\epsilon, \delta}$ we have
  \begin{equation}\label{e:LsBounds}
    p_{\min}^{1/d} \Lambda_{p, \delta} \epsilon
      \leq \ell_s
      \leq \Lambda_{p, \delta} \epsilon\,.
  \end{equation}
  If $\delta < 2$, then $\Lambda_{p, \delta} \geq 1$, and we may apply Lemma~\ref{l:pcMix}.
  Using~\eqref{e:HLowerBound}, \eqref{e:HSummable}, and~\eqref{e:LambdaPDelta} in~\eqref{e:fullCylLpEstimate} immediately implies~\eqref{e:fullCylLpEstimate2}.
  If~$\delta \geq 2$ then~\eqref{e:fullCylLpEstimate2} follows directly from the triangle inequality and the fact that~$T_*$ and~$U_*$ are contractions.
\end{proof}

Momentarily postponing the proof of Lemma~\ref{l:pcMix}, we prove Theorem~\ref{t:TMixLower}.

\begin{proof}[Proof of Theorem~\ref{t:TMixLower}]
  For any~$\delta \in (0, 1)$, define
  \begin{equation}
    \delta' = 1 + \delta\,,
    \quad
    \delta'' = 1 - \delta\,,
  \end{equation}
  and let $\mathcal S = \mathcal S_{\epsilon,\delta''}$ be defined by~\eqref{e:SepDelta} (with~$\delta = \delta''$).
  Let~$i \in \mathcal I$ be such that $p_i = p_{\max}$, and choose
  \begin{equation*}
    N = \ceil[\Big]{ \frac{ d \ln \epsilon \Lambda_{1, \delta''} }{\ln p_{\max}} }\,,
    \quad
    N_1 = \ceil[\Big]{ \frac{ \ln (1 - \delta') }{\ln p_{\max}} }\,,
    \quad
    s = \underbrace{(i, i, \dots, i)}_{N \text{ times}} \,,
    \quad
    t = \underbrace{(i, i, \dots, i)}_{N_1 \text{ times}}  \,.
  \end{equation*}
  Note~$s \in \mathcal S_{\epsilon,\delta}$, and so by Corollary~\ref{c:TStarUstar} we note
  \begin{equation*}
    \norm{T_*^{N - N_1} \bm I_s - \bm I_t}_{L^1}
      \leq \norm{T_*^{N - N_1} \bm I_s - U_*^{N - N_1} \bm I_s}_{L^1}
      \leq \delta'' \,.
  \end{equation*}
  Also, by choice of~$N_1$ and~$t$,
  \begin{equation*}
    \norm{\bm I_t - 1}_{L^1}
      \geq 1 - p_{\max}^{\abs{t}}
      \geq \delta'\,.
  \end{equation*}
  % EL: we actually get 2 (1-p_max^|t|) on the RHS, but I guess this doesn't really matter
  Thus, by the triangle inequality
  \begin{equation*}
    \norm{T_*^{N - N_1} \bm I_s  - 1}_{L^1}
      \geq \norm{\bm I_t  - 1}_{L^1}
	- \norm{T_*^{N - N_1} \bm I_s  - \bm I_t}_{L^1}
      \geq \delta' - \delta'' = 2\delta\,.
  \end{equation*}

  Consequently, if we choose $X_0$ to have density~$\bm I_s$, then
  \begin{equation}\label{e:DistXnMN1}
    \norm{\dist(X_{N - N_1}) - \pi}_\TV = \frac{1}{2} \norm{T_*^{N - N_1} \bm I_s - 1}_{L^1} \geq \delta
  \end{equation}
  We clarify that the factor~$1/2$ above arises from the commonly used normalization convention
  \begin{equation}\label{e:TVdef}
    \norm{\mu - \nu}_\TV \defeq \sup_{A \subseteq \T^d} \frac{\abs{\mu(A) - \nu(A)}}{2}\,.
  \end{equation}
  The lower bound~\eqref{e:DistXnMN1} immediately implies~$\tmix(\delta) \geq N - N_1$ and the choice of~$N, N_1$ implies~\eqref{e:TMixLower} as desired.
\end{proof}

\subsection{Mixing piecewise constant functions (Lemma \ref{l:pcMix})}\label{s:cylMixLpProof}

We begin by estimating the mass that leaks out of cylinder sets due to the action of the noise.

\begin{lemma}[Convolution estimates]\label{l:convLp}
  Suppose~$K_\epsilon$ satisfies~\eqref{e:FirstMomentKep}.
  For any $p \in [1, \infty)$,  $\epsilon > 0$, $s \in \mathcal T$ we have
  \begin{equation*}
    \norm{\bm I_s * K_\epsilon - \bm I_s}_{L^p} \leq
    \frac{(\epsilon \Per( \cyl_s ))^{1/p}}{\pi( \cyl_s )}.
  \end{equation*}
\end{lemma}
\begin{remark*}
  If $K_\epsilon$ is supported in $B(0, \epsilon)$, then certainly $K_\epsilon$ satisfies~\eqref{e:FirstMomentKep}.
\end{remark*}
\begin{proof}
  This is a standard result that is true in more generality (see for instance Lemma 3.24 in~\cite{AmbrosioFuscoEA00}).
  For convenience we present a simple proof here.
  Note that
  \begin{equation*}
    \int_{x \in \T^d} \abs{ \one_{\cyl_s}(x - y) - \one_{\cyl_s}(x) }^p \, dx
      \leq d(0, y) \Per(\cyl_s)\,,
  \end{equation*}
  where~$d(0, y)$ denotes the torus distance between~$y$ and~$0$.
  Thus, by Jensen's inequality and~\eqref{e:FirstMomentKep} we see
  \begin{align*}
    \norm{\one_{\cyl_s} * K_\epsilon - \one_{\cyl_s}}_{L^p}^p
      &\leq \int_{y \in B(0, \epsilon)} \int_{x \in \T^d} K_\epsilon(y) \abs{ \one_{\cyl_s}(x - y) - \one_{\cyl_s}(x) }^p \, dx \, dy
    \\
      &\leq \int_{y \in B(0, \epsilon)} d(0, y) K_\epsilon(y) \Per(\cyl_s)  \, dy
      = \epsilon \Per(\cyl_s)\,.
    \qedhere
  \end{align*}
\end{proof}

Using the convolution estimates above, we can now prove the desired \(L^p\) estimate.
\begin{proof}[Proof of Lemma \ref{l:pcMix}]
  Using~\eqref{e:evolutionIdentity}, we note that if~$f_0$ is given by~\eqref{e:f01}, we have
  \begin{equation*}
    U_* f_0
      = U_* \sum_{s \in \mathcal S} c_0(s) \bm I_s \pi(\cyl_s)
      = \sum_{s \in \mathcal S} c_0(s) \bm I_{\sigma s} \pi( \cyl_s )\,.
  \end{equation*}
  This implies
  \begin{equation*}
    f_1 \defeq T_* f_0
      = \sum_{s \in \mathcal I } c_0(s) \bm I_{\sigma s} \pi( \cyl_s )
	+ \tilde f_1
     = U_{\ast} f_0
	+ \tilde f_1\,,
  \end{equation*} 
  where
  \begin{equation*}
    \tilde f_1 \defeq
      T_* f_0 - U_* f_0
      = \sum_{s \in \mathcal I} c_0(s) (\bm I_{\sigma s}*K_\epsilon - \bm I_{\sigma s}) \pi(\mathcal C_s)\,.
  \end{equation*}

  For any $t \in \mathcal I$, we partition \(\mathcal S\) into the sets $\set{\mathcal S^t \st t \in \mathcal I}$, where
  \[
      \mathcal S^t \defeq \set{ s \in \mathcal S \st \cyl_s \subseteq \cyl_t } \,.
  \]
  Grouping the error terms associated to each \(\mathcal S^t\), we can write
  \begin{equation}\label{e:ftTildeDef}
    \tilde f_1 = \sum_{t \in \mathcal I} \tilde f_{1, t}
    \quad\text{where}\quad
    \tilde f_{1, t} \defeq \sum_{s \in \mathcal S^t} c_0(s)
      \paren[\big]{ \bm I_{\sigma s} * K_\epsilon - \bm I_{\sigma s}} \, \pi( \cyl_s )
      \,.
  \end{equation}

  We now claim 
  \begin{equation}\label{e:TiF1tp}
    \norm{\tilde f_{1, t}}^p_{L^p}
      \leq C_1^{p/p'} \sum_{s \in \mathcal S^t} \abs{c_0(s)}^p \norm{ \bm I_{\sigma s} * K_\epsilon - \bm I_{\sigma s} }_{L^p}^p \pi(\cyl_s)^p\,,
  \end{equation}
  where~$C_1$ is the constant in~\eqref{e:C1Def}.
  If~$p = 1$, then $C_1^{p / p'} = 1$, and so clearly~\eqref{e:TiF1tp} holds.

  If $p > 1$, we recall that~$\supp(K_\epsilon) \subseteq B(0, \epsilon)$ by assumption.
  We will use this to show the support of each term in~\eqref{e:ftTildeDef} intersects the support of at most~$C_1$ other terms.
  This will imply~\eqref{e:TiF1tp} as claimed.

  To count the number of terms in~\eqref{e:ftTildeDef} with intersecting support, note that $\set{\cyl_s \st s \in \mathcal S^t}$ partitions~$\mathcal S^t$.
  Hence the sets $\set{ \cyl_{\sigma s} \st s \in \mathcal S^t }$ partition the torus.
  We claim now that for any~$s \in \mathcal S^t$, the number of $s' \in \mathcal S^t$ such that $B(\cyl_{\sigma s}, \epsilon )\cap B(\cyl_{\sigma s'}, \epsilon) \neq \emptyset$ is at most the constant~$C_1$ in equation~\eqref{e:C1Def}.
  To see this, we note that~\eqref{e:Leps} and the fact that~$L \geq 1$  imply $\diam(\cyl_{\sigma s'}) \geq 2 \epsilon$ for every $s' \in S^t$.
  Moreover every cylinder set is an axis-aligned cube, and the sets $\set{ \cyl_{\sigma s'} \st s' \in \mathcal S^t}$  partition the torus.
  Consequently, if for some $s, s' \in \mathcal S^t$, the sets $B(\cyl_{\sigma s}, \epsilon)$ and $B(\cyl_{\sigma s'}, \epsilon)$ intersect, the closures of $\cyl_{\sigma s}$ and \(\cyl_{\sigma s'}\) themselves must intersect.

  Given~\eqref{e:Leps} and the fact that $\diam(\cyl_{s'}) \geq p_{\min}^{1/d} \diam(\cyl_{\sigma s'})$, the number of disjoint sets intersecting a given face of~\(\cyl_{\sigma s}\) is at most $(2 + p_{\min}^{-1/d})^{d-1}$.
  Since there are $2d$ faces, there are at most~$2d (2 + p_{\min}^{-1/d})^{d-1} = C_1$ cylinder sets~$\cyl_{\sigma s'}$ such that the neighborhoods $B(\cyl_{\sigma s}, \epsilon)$ and $B(\cyl_{\sigma s'}, \epsilon)$ intersect.

  Since \(\supp(K_\epsilon) \subseteq B(0, \epsilon)\) by assumption, \(\supp (\bm I_{\sigma s} * K_\epsilon) \subseteq B(\cyl_{\sigma s},\epsilon)\), and hence the support of each term in~\eqref{e:ftTildeDef} intersects the support of at most~$C_1$ other terms.
  This shows~\eqref{e:TiF1tp} holds even when~$p > 1$.
  \smallskip

  Using~\eqref{e:TiF1tp} with Lemma~\ref{l:convLp}, we have
  \begin{align*}
    \norm{\tilde f_{1, t}}^p_{L^p}
      &\leq C_1^{p/p'} \sum_{s \in \mathcal S^t} \abs{c_0(s)}^p \norm{ \bm I_{\sigma s} * K_\epsilon - \bm I_{\sigma s} }_{L^p}^p \pi(\cyl_s)^p
    \\
      &\leq \epsilon C_1^{p/p'} H(\sigma (\mathcal S^t) ) \sum_{s \in \mathcal S^t} \abs{c_0(s)}^p \pi(\cyl_{\sigma s} )^{-p/p'} \pi(\cyl_s)^p
    \\
      &\leq \epsilon C_1^{p/p'} H(\sigma(\mathcal S^t)) \pi(\cyl_t)^{p/p'}
	\sum_{s \in \mathcal S^t} \abs{c_0(s)}^p \pi(\cyl_s) \,,\\
    &\leq \epsilon C_1^{p/p'} H(\sigma(\mathcal S)) \pi(\cyl_t)^{p/p'}
	\sum_{s \in \mathcal S^t} \abs{c_0(s)}^p \pi(\cyl_s) \,.
  \end{align*}
  This implies
  \begin{align}
    \nonumber
    \sum_{t \in \mathcal I} \norm{\tilde f_{1, t}}_p
      &\leq \epsilon^{1/p} C_1^{1/p'}
       H( \sigma( \mathcal S ) )^{1/p}
	\sum_{t \in \mathcal I} \pi(\cyl_t)^{1/p'}
	\paren[\Big]{ \sum_{s \in \mathcal S^t} \abs{c_0(s)}^p \pi(\cyl_s) }^{1/p}
    \\
    \label{e:tmpf1}
      &\leq \epsilon^{1/p} C_1^{1/p'}
       H( \sigma( \mathcal S ) )^{1/p}
	\paren[\Big]{ \sum_{s \in \mathcal S } \abs{c_0(s)}^p \pi(\cyl_s) }^{1/p} \,.
  \end{align}
  Since
  \begin{equation*}
    \norm{f_0}_{L^p}^p
      = \sum_{s \in \mathcal S} \abs{c_0(s)}^p \norm{\one_{\cyl_s}}_{L^p}^p
      = \sum_{s \in \mathcal S} \abs{c_0(s)}^p \pi(\cyl_s)\,,
  \end{equation*}
  the inequality~\eqref{e:tmpf1} implies
  \begin{equation*}
    \norm{\tilde f_1}_{L^p} \leq 
    \epsilon^{1/p} C_1^{1/p'} H( \sigma \mathcal S )^{1/p} \norm{f_0}_{L^p}\,.
  \end{equation*}
  \smallskip

  Now for $n \geq 2$ we inductively define the error term \(\tilde f_n\) via the identity
  \begin{align*}
    f_{n}
      = T_* f_{n-1}
      &= U_*^n f_0 + \sum_{i = 1}^{n-1} T_*^i \tilde f_{n-i} + \tilde f_n
  \end{align*}
  and write
  \begin{equation*}
    \tilde f_{n} \defeq \sum_{t \in \mathcal I^{n}} \tilde f_{n, t}
    \quad\text{where}\quad 
    \tilde f_{n, t} = \sum_{s \in \mathcal S_t} c_0(s)
      \paren[\big]{ \bm I_{\sigma^n s} * K_\epsilon - \bm I_{\sigma^n s}} \, \pi( \cyl_s )
      \,.
  \end{equation*}
  Using the same argument as above, we see
  \begin{equation*}
    \norm{\tilde f_n}_{L^p} \leq 
      \epsilon^{1/p} C_1^{1/p'} H( \sigma^n \mathcal S )^{1/p} \norm{f_0}_{L^p}\,,
  \end{equation*}
  for all \(1 \le n < \max_{s \in \mathcal S} \abs{s}\).

  Since \(T_\ast\) is an \(L^p\) contraction, we have
  \begin{equation*}
    \norm{T_*^N f_0 - U_*^N f_0}_{L^p}
    \leq \sum_{i=1}^n \norm{\tilde f_n}_{L^p} 
    \leq \epsilon^{1/p} C_1^{1/p'}
	\sum_{n = 1}^{N} H( \sigma^n \mathcal S )^{1/p} \norm{f_0}_{L^p}\,.
  \end{equation*}
  for all~$N < \max_{s \in \mathcal S} \abs{s}$.
  This proves~\eqref{e:fullCylLpEstimate} for all $N < \max_{s \in \mathcal S} \abs{s}$.

  For $n \geq \max_{s \in \mathcal S}$ we note
  \begin{equation*}
    \bm I_{\sigma^n s} * K_\epsilon - \bm I_{\sigma^n s}
      = \bm I_{\zerotuple} * K_\epsilon - \bm I_{\zerotuple}
      = 0\,,
  \end{equation*}
  and hence~$\tilde f_n = 0$.
  Since~$H(\zerotuple) = 0$ by convention, this implies~\eqref{e:fullCylLpEstimate} for all~$n \in \N$.
  \end{proof}

\section{Relation between the dissipation time and mixing time.}\label{s:TMixTDis}

The upper bound for the dissipation time~\eqref{e:tmix-leq-tdis} is a general fact and requires nothing but the Markov property.
The proof is very similar to the proof of Proposition 2.2 in~\cite{IyerZhou22}.
Since the proof is short and elementary, we present it here for the readers convenience.
\begin{proof}[Proof of the upper bound~\eqref{e:tdis-leq-tmix} in Proposition~\ref{p:tmixTdis}]
  Let~$\rho_n(x, y)$ be the transition density of the process~$X$ after~$n$ time steps, and~$\theta_0 \in \dot L^2$.
  Then we note
  \begin{equation*}
    \theta_n(x)
      \defeq \E^x \theta_0(X_n)
      = \int_{\T^d} \rho_n(x, y) \theta_0(y) \, \pi(dy) \,.
  \end{equation*}
  Since $\theta_0$ has mean~$0$, this implies
  \begin{equation*}
    \theta_n(x)
    = \int_{\T^d} (\rho_n(x, y) - 1) \theta_0(y) \, \pi(dy)\,,
  \end{equation*}
  and hence
  \begin{align*}
    \norm{\theta_n}_{L^2}^2 &\leq
    \paren[\Big]{ \int_{\T^d \times \T^d}
    \abs{\rho_n(x, y) - 1} \, \pi(dy) \, \pi(dx) }
    \cdot
    \\
      &\qquad\cdot
    \paren[\Big]{ \int_{\T^d \times \T^d} \theta_0(y)^2 \paren{ \rho_n(x, y) + 1 } \, \pi(dx) \, \pi(dy) }
    \\
      &\leq 2 \norm{\theta_0}_{L^2}^2
      \sup_{x \in \T^d} \int_{\T^d} \abs{\rho_n(x, y) - 1} \, dy\,.
  \end{align*}
  The last inequality above followed from the fact that $\rho_n$ is nonnegative, and leaves the measure~$\pi$ invariant.
  Notice that the right hand side is at most \(4 \norm{\theta_0}_{L^2}^2 \norm{ \dist(X_n) - \pi }_\TV\), where we recall the total variation norm is normalized by a factor of~$\frac{1}{2}$ as in~\eqref{e:TVdef}.
  Thus, when~$n \geq \tmix(\delta^2 / 4)$, we have~$\norm{\theta_n}_{L^2}^2 \leq \delta^2 \norm{\theta_0}_{L^2}$, proving~\eqref{e:tdis-leq-tmix} as claimed.
\end{proof}

In order to prove the lower bound for the dissipation time~\eqref{e:tmix-leq-tdis}, it is convenient to introduce the ``backward'' operator~$T$, defined by
\begin{equation}\label{e:TDef}
  T f(x) = \E^x f(X_1) = \int_{\T^d} \rho(x, y) f(y) \, \pi(dy)\,,
\end{equation}
where we recall~$\rho$ is the transition kernel of the process~$X$.
From~\eqref{e:TStarDef} we see that for any~$f, g \in L^2$ we have
\begin{equation}\label{e:TTstarDual}
  \ip{Tf, g}
    = \int_{\T^d \times \T^d} \rho(x, y) f(y) g(x) \, \pi(dy) \, \pi(dx)
    = \ip{f, T_*g}\,,
\end{equation}
where~$\ip{\cdot, \cdot}$ denotes the standard~$L^2$ inner-product.

Moreover, if~$\rho_n$ is the $n$-step transition kernel of~$X$, then by the Markov property and induction we see
\begin{align}
  T^n f (x)
    &= T^{n-1} T f(x)
    = \int_{\T^d \times \T^d} \rho_{n-1}(x, y) \rho(y, z) f(z) \, \pi(dz) \, \pi(dy)
  \\
    &= \int_{\T^d} \rho_n(x, z) f(z) \, \pi(dz)
    = \E^x f(X_n)\,.
\end{align}
In light of this, we may rewrite the dissipation time~$\tdis(\delta)$ as
\begin{equation}
  \tdis(\delta) = \min \set[\big]{
    n \st \norm{ T^n f }_{L^2} \leq  \delta \norm{f}_{L^2} 
      \text{ for all } f \in \dot L^2
  }\,.
\end{equation}
We claim~$T$ can be replaced with~$T_*$ in the above definition.
That is, we claim
\begin{equation}\label{e:TDisTStar}
  \tdis(\delta) = \min \set[\big]{
    n \st \norm{ T_*^n f }_{L^2} \leq  \delta \norm{f}_{L^2} 
      \text{ for all } f \in \dot L^2
  } \,.
\end{equation}
To see this let $m = \tdis^*(\delta)$ be the right hand side of~\eqref{e:TDisTStar}.
Using~\eqref{e:TTstarDual} we see
\begin{equation}
  \norm{T^m f}_{L^2}^2 = \ip{T^m f, T^m f}
    = \ip{f, T_*^n T^m f}
    \leq \norm{f}_{L^2} \norm{T_*^m T^m f}_{L^2}
    \leq \delta \norm{f}_{L^2} \norm{T^m f}_{L^2} \,,
\end{equation}
and hence
\begin{equation}
  \norm{T^m f}_{L^2} \leq \delta \norm{f}_{L^2}\,.
\end{equation}
This implies~$\tdis(\delta) \leq m = \tdis^*(\delta)$.
The reverse inequality follows by symmetry, showing~\eqref{e:TDisTStar}.
\medskip

We will now prove the lower bound~\eqref{e:tmix-leq-tdis} by using the identity~\eqref{e:TDisTStar}
\begin{proof}[Proof of the lower bound~\eqref{e:tmix-leq-tdis} in Proposition~\ref{p:tmixTdis}]
  Since~$K_\epsilon \in L^1$, we note that
  \begin{equation*}
    \P( X_1 \in dx ) = f_1(x) \pi(dx)\,,
  \end{equation*}
  for some probability density function~$f_1 \in L^1$.
  By Young's inequality,
  \begin{equation*}
    \norm{T_* (f_1 - 1)}_{L^2} \leq \norm{K_\epsilon}_{L^2} \norm{U_* (f_1 - 1)}_{L^1}
      \leq \frac{2 \bm K }{\epsilon^{d/2} }\,.
  \end{equation*}
  Thus
  \begin{equation}\label{e:Xn2MinusPi}
    \norm{\dist(X_{n+2}) - \pi}_\TV
    = \frac{\norm{T_*^{n+1} f_1 - 1}_{L^1}}{2}
    \leq \frac{\norm{T_*^{n+1} (f_1 - 1)}_{L^2}}{2}\,.
  \end{equation}
  Notice for any~$g \in \dot L^2$ and~$n = m \tdis(\delta)$ we have
  \begin{equation}
    \norm{T_*^n(g)}_{L^2} \leq \delta^m \norm{g}_{L^2} \,.
  \end{equation}
  Thus choosing
  \begin{equation*}
    n \geq \log_\delta \paren[\Big]{ \frac{\delta' \epsilon^{d/2}}{\bm K} } \tdis(\delta)
  \end{equation*}
  the right hand side of~\eqref{e:Xn2MinusPi} is at most~$\delta'$.
  This implies~\eqref{e:tmix-leq-tdis}, concluding the proof.
\end{proof}
\section{Dissipation Time Bounds.}\label{s:dissipationTimeBounds}
\subsection{The upper bound on the dissipation time.}\label{s:DtimeUpperBd}
Throughout this section we will make the same assumptions as Theorem~\ref{t:pl-dis-time}.
That is, we will assume Assumption~\ref{a:Cubes} holds, and~$\supp(K_\epsilon) \subseteq B(0, \epsilon)$.

The main idea behind the proof is to use the noise to dissipate high frequencies, and the dynamics of~$\varphi$ to mix the low frequencies.
While dissipating high frequencies with the noise is straightforward, showing that the low frequency data gets mixed requires a little more care.
We state this as our first lemma.
\begin{lemma}\label{l:LowFreqDecay}
  Let $p \in [1, \infty)$, $\delta \in (0, 1)$.
  There exists an explicit constant~$B_{p,\delta}$ such that if~$f \in \dot L^p$ and
  \begin{equation}\label{e:GradFLpSmall}
    \norm{\grad f}_{L^p}
      \leq \frac{B_{p,\delta}}{\epsilon} \norm{f}_{L^p}\,,
  \end{equation}
  then
  \begin{equation*}
    \norm{T_*^n f}_{L^p} \leq \frac{\delta}{2} \norm{f}_{L^p}\,,
    \quad\text{for all }
    n \geq \frac{d \ln (\epsilon \Lambda_{p, \delta/4})}{\ln p_{\max} }
    \,.
  \end{equation*}
\end{lemma}
\begin{remark}
  Let~$X_0 \sim \mu$, $f_1 = K_\epsilon * U_* \mu$, and note $\norm{\grad f_1 - 1}_{L^1} \leq 2 \norm{\grad K_\epsilon}_{L^1}$.
  If
  \begin{equation}\label{e:KepSmall}
    2 \norm{\grad K_\epsilon}_{L^1} < B_{p,\delta}\,,
  \end{equation}
  it is possible to use Lemma~\ref{l:LowFreqDecay} to show that for some~$\delta \in (0, 1)$ we have the expected sharp upper bound stated in~\eqref{e:TMixDeltaSharpUpper}.
  However, the constant~$B_{p, \delta}$ is related to the Poincar\'e constant of cylinder sets and can be computed explicitly.
  Even the indicator function of a ball does not satisfy~\eqref{e:KepSmall}, and it may not be possible to find even one kernel (that satisfies~\eqref{e:FirstMomentKep}) for which~\eqref{e:KepSmall} holds.
\end{remark}

Here~$\Lambda_{p, \delta/4}$ is defined by~\eqref{e:LambdaPDelta} (in Corollary~\ref{c:TStarUstar}).
Postponing the proof of Lemma~\ref{l:LowFreqDecay}, we prove the upper bound in Theorem~\ref{t:pl-dis-time}.
While Lemma~\ref{l:LowFreqDecay} holds for any~$p \in [1, \infty)$, we are presently only able to apply it when~$p = 2$ as our proof of Theorem~\ref{t:pl-dis-time} relies on orthogonal projections in frequency space.

\begin{proof}[Proof of the upper bound in Theorem~\ref{t:pl-dis-time}]
  Define the Fourier projections~$\Plow$ and $\Phigh$, which project functions onto the low and high frequency spaces respectively.
  Explicitly, define
  \begin{equation*}
    (\Plow f)^\wedge(k)
      = \sum_{\abs{k} \leq \frac{B}{2\pi \epsilon}} \hat f(k)
    \qquad\text{and}\qquad
    \Phigh = I - \Plow\,,
  \end{equation*}
  where~$B = B_{2, \delta}$ is the constant appearing in~\eqref{e:GradFLpSmall}, and $\hat f(k)$ is the $k$-th Fourier coefficient of~$f$.

  The reason for defining~$\Plow$ and~$\Phigh$ as above is that functions in the range of~$\Plow$ have well controlled gradients, and can be mixed using Lemma~\ref{l:LowFreqDecay}.
  On the other hand, functions in the range of~$\Phigh$ only have high frequencies and are rapidly mixed by~$\Phigh$.
  More precisely, for any~$g \in \dot L^2$, we have
  \begin{equation*}
    \norm{\grad \Plow g}_{L^2} \leq \frac{B}{\epsilon} \norm{\Plow g}_{L^2}\,,
  \end{equation*}
  and hence Lemma~\ref{l:LowFreqDecay} can be applied to $\Plow g$.

  For high frequencies,
  \begin{equation*}
    \norm{K_\epsilon * \Phigh g}_{L^2}
      = \paren[\Big]{ \sum_{\abs{k} \geq \frac{B}{2\pi \epsilon}} \abs{ \hat K_\epsilon}^2 (\Phigh g)^\wedge(k)^2}^{1/2}
      \leq (1 - \chi) \norm{\Phigh g}_{L^2} \,,
  \end{equation*}
  where~$\chi = \chi(\delta)$ is defined by
  \begin{equation*}
    \chi = \sup_{\abs{k} \geq \frac{B}{2\pi \epsilon}} \abs{\hat K_\epsilon(k)}\,.
  \end{equation*}
  Recall that assumption~\eqref{e:KHatDecay} guarantees $\chi < 1$.
  \smallskip

  We will now show
  \begin{equation}\label{e:tDisStarDelta}
    \tdis(\delta) \leq N + N_1 + 1
  \end{equation} 
  where
  \begin{equation*}
    N = \ceil[\Big]{ \frac{d \ln (\epsilon \Lambda_{2, \delta/4})}{\ln p_{\max} } }\,,
    \quad\text{and}\quad
    N_1 = \ceil[\Big]{
      \frac {\ln \delta}
	{\ln\paren[\Big]{1 - \frac{\delta^2(1 - (1 - \chi)^2)}{4 (1 - \chi)^2} }}
      }\,,
  \end{equation*}
  provided
  \begin{equation}
    \delta^2 \leq \frac{2(1 - \chi)^2}{1 - (1 - \chi)^2}\,.
  \end{equation}
  We divide the proof into two cases.

  \restartcases
  \case
  Suppose for some $n \leq N_1$ we have
  \begin{equation}\label{e:PHighSmall}
    \norm{\Phigh U_* f_n}_{L^2} \leq \frac{\delta}{2(1 - \chi)} \norm{f_n}_{L^2}\,.
  \end{equation}
  Let~$\mathcal K_\epsilon g = K_\epsilon * g$ denote the convolution operator, and observe
  \begin{align*}
    \norm{T_*^{N+1}f_n}_{L^2} &\leq
      \norm{T_*^N \mathcal K_\epsilon \Plow U_* f_n}_{L^2}
      + \norm{T_*^N \mathcal K_\epsilon \Phigh U_* f_n}_{L^2}
    \\
      &\leq \frac{\delta}{2} \norm{\Plow U_* f_n}_{L^2}
	+ (1 - \chi) \norm{\Phigh U_*f_n}_{L^2} 
      \leq \delta \norm{f_n}\,.
  \end{align*}
  This implies
  \begin{equation}\label{e:fNN11}
    \norm{f_{N + N_1 + 1}}_{L^2} \leq \delta \norm{f_0}_{L^2} \,.
  \end{equation}

  \case
  Suppose now~\eqref{e:PHighSmall} does not hold for any $n \leq N_1$.
  In this case we must have
  \begin{align*}
    \norm{f_{n+1}}_{L^2}^2
      &= \norm{\mathcal K_\epsilon \Phigh U_* f_n}_{L^2}^2
	+ \norm{\mathcal K_\epsilon \Plow U_* f_n}_{L^2}^2
    \\
      &\leq (1 - \chi)^2 \norm{\Phigh U_* f_n}_{L^2}^2
	+ \norm{U_* f_n}_{L^2}^2 - \norm{\Phigh U_* f_n}_{L^2}^2 
    \\
      &\leq \paren[\Big]{1 - \frac{\delta^2(1 - (1 - \chi)^2)}{4 (1 - \chi)^2} } \norm{f_n}\,.
  \end{align*}
  Iterating this and using the definition of~$N_1$ immediately implies
  \begin{equation*}
    \norm{f_{N_1}}_{L^2} \leq \delta \norm{f_0}_{L^2}\,,
  \end{equation*}
  and hence~\eqref{e:fNN11} also holds in this case.
  \smallskip

  Thus in either case~\eqref{e:fNN11} holds, and using~\eqref{e:TDisTStar} this implies~\eqref{e:tDisStarDelta} as claimed.
  By choice of~$N$ and~$N_1$ this proves the upper bound in~\eqref{e:plDisTime} as claimed.
\end{proof}

\subsection{Mixing low frequency data (Lemma \ref{l:LowFreqDecay}).}\label{s:LowFreqDecay}

Note if~\eqref{e:GradFLpSmall} holds, then we can approximate~$f$ by a function that is piecewise constant on small cylinder sets. 
For such functions we may use Lemma~\ref{l:pcMix} (or Corollary~\ref{c:TStarUstar}) to show that~$T_*^n f$ is close to~$U_*^n f$.
Of course, if~$f \in \dot L^2$ is in the form~\eqref{e:f01}, and~$N \geq_{s \in \mathcal S} \max \abs{s}$, then
\begin{equation*}
  U_*^N f = \sum_{s \in \mathcal S} c_0(s) \pi(\cyl_s) = \int_{\T^d} f \, d\pi = 0\,.
\end{equation*}
so the bound~\eqref{e:fullCylLpEstimate2} in Corollary~\ref{c:TStarUstar} will ensure~$\norm{T_*^N f_0}_{L^2}$ is small.
This is the strategy we will use to prove Lemma~\ref{l:LowFreqDecay}.

\begin{proof}[Proof of Lemma~\ref{l:LowFreqDecay}]
  Let~$\Lambda = \Lambda_{p, \delta/4}$ be defined by~\eqref{e:LambdaPDelta}, and~$\mathcal S = \mathcal S_{\epsilon, \delta}$ be as in~\eqref{e:SepDelta}.
  Recall, for every~$s \in \mathcal S$ the side lengths~$\ell_s$ are bounded by~\eqref{e:LsBounds}.
  Moreover since $p_{\min}^{1/d} \ell{\sigma s} \leq \ell_{s} \leq p_{\max}^{1/d} \ell_{\sigma s}$ we must have
  \begin{equation}\label{e:tupleLenBounds}
    \frac{d \ln (\epsilon \Lambda_{p, \delta/4} )}{\ln p_{\min} }\le \abs{s} \le   \frac{d \ln (\epsilon \Lambda_{p, \delta/4})}{\ln p_{\max} }   \quad\text{for all } s \in \mathcal S\,.
  \end{equation}
  Choose
  \begin{equation*}
    N \defeq \max_{s \in \mathcal S} \abs{s} \leq \frac{d \ln (\epsilon \Lambda_{p, \delta/4})}{\ln p_{\max} }\,.
  \end{equation*}

  Now let~$B_{p, \delta}$ be a constant that will be chosen shortly, and suppose~$f \in \dot L^p$ satisfies~\eqref{e:GradFLpSmall}.
  Define
  \begin{equation*}
    c_s = \frac{1}{\pi(\mathcal C_s)} \int_{\cyl_s} f \, d\pi\,,
    \quad
    f_0 = \sum_{s \in \mathcal S} c_s \one_{\cyl_s} \,,
    \quad
    \tilde f_0 = f - f_0\,.
  \end{equation*}
  Since $U_*^N f_0 = 0$ (by choice of~$N$), Corollary~\ref{c:TStarUstar} implies
  \begin{equation*}
    \norm{T_*^N f_0}_{L^p}
      = \norm{T_*^N f_0 - U_*^N f_0}_{L^p}
      \leq \frac{\delta}{4} \norm{f_0}_{L^p}
      \leq \frac{\delta}{4} \norm{f}_{L^p}\,.
  \end{equation*}

  To bound~$T_*^N \tilde f_0$ observe that the Poincar\'e inequality,
  \begin{align*}
    \norm{\tilde f_0}_{L^p}^p
      &= \norm[\Big]{f - \sum_{s \in \mathcal S} c_s \one_{\cyl_s} }_{L^p}^p
      = \sum_{s \in \mathcal S} \norm{f - c_s}_{L^p(\cyl_s)}^p
    \\
      &\leq \diam(\cyl_s)^p \sum_{s \in \mathcal S} \norm{\grad f}_{L^p(\cyl_s)}^p
      \leq (\epsilon \Lambda_{p, \delta/4} \sqrt{d} )^p \norm{\grad f}_{L^p}^p\,.
  \end{align*}
  Thus if we choose
  \begin{equation*}
    B_{p, \delta} \defeq \frac{\delta}{4 \Lambda_{p, \delta/4} \sqrt{d}}\,,
  \end{equation*}
  the assumption~\eqref{e:GradFLpSmall} will imply
  \begin{equation*}
    \norm{T_*^N \tilde f_0}_{L^p}
      \leq \norm{\tilde f_0}_{L^p}
      \leq \frac{\delta}{4} \norm{f}_{L^p}\,.
  \end{equation*}

  Consequently, for any~$n \geq N$,
  \begin{equation*}
    \norm{T_*^n f}_{L^p}
      \leq \norm{T_*^N f}_{L^p}
      \leq \norm{T_*^N f_0}_{L^p}
	+ \norm{T_*^N \tilde f_0}_{L^p}
      \leq \frac{\delta}{2}\,,
  \end{equation*}
  concluding the proof.
\end{proof}

\subsection{Lower bounds on the dissipation time.}
The strategy to prove the lower bound in Theorem~\ref{t:pl-dis-time} is similar to the strategy used to prove Theorem~\ref{t:TMixLower}.
The main difference is that for Theorem~\ref{t:TMixLower} we could choose our initial distribution to be uniformly distributed on a small cylinder set.
We can not do this for Theorem~\ref{t:pl-dis-time}.
If we do, one iteration of~$T_*$ will spread it on a larger cylinder set,
and decreases the~$L^p$ norm by a constant factor for every $p \in (1, \infty)$.
Thus our initial distribution has to be spread over many small cylinder sets.
This now requires us to control the smallest cylinder set  our initial data is supported on, and results in our bounds being of order $d \abs{\ln \epsilon} / p_{\min}$, and not $d \abs{\ln \epsilon} / p_{\max}$ as was the case with Theorem~\ref{t:TMixUpper}.

\begin{proof}[Proof of the lower bounds in Theorem~\ref{t:pl-dis-time}]
  Let $\delta' \in (0, 1)$ be a small number that will be chosen shortly, and \(\mathcal S\) be the same partition that was used in the proof of Lemma~\ref{l:LowFreqDecay} with \(\Lambda = \Lambda_{2, \delta'}\), and set \(N= \min_{s \in \mathcal S} \abs{s}\).
  By~\eqref{e:tupleLenBounds}, we know that 
  \[
    N 
    \ge \dfrac{d \ln (\Lambda \epsilon)}{\ln p_{\min}}
    = \dfrac{d \ln ( \epsilon)}{\ln p_{\min}} + \dfrac{d \ln ( \Lambda)}{\ln p_{\min}}
    = \dfrac{d \ln ( \epsilon)}{\ln p_{\min}} - C (1 - \ln \delta' ) \,,
  \]
  for some explicit constant \(C > 0 \) that can be computed from~\eqref{e:LambdaPDelta}.

  Consider the function given by
  \[
    f_0(x) = \begin{cases}
      1 & \varphi^{N}(x) \in E_1 \\
      p_1/p_2 & \varphi^{N}(x) \in E_2 \\
      0 & \text{otherwise}
    \end{cases}\,,
  \]
  which is of the form~\eqref{e:f01} by our choice of \(N\).
  Note that
  \[
    U_\ast^{N-1}f_0(x) = \begin{cases}
      1 & \varphi(x) \in E_1 \\
      p_1/p_2 & \varphi(x) \in E_2 \\
      0 & \text{otherwise}
    \end{cases}\,.
  \]
  Thus \(f_0\) and \(U_\ast^{N-1}f_0\) have the same distribution function, and hence must have the same \(L^2\) norm.
  Applying Corollary~\ref{c:TStarUstar}, we immediately see that
  \begin{align*}
    \norm{T_\ast^N f_0}_{L^2} 
	&\ge \norm{U_\ast^N f_0}_{L^2} - \norm{T_\ast^N f_0 -U_\ast^N f_0}_{L^2}
	\ge \paren{ 1 - \delta'} \norm{f_0}_{L^2}
	= \delta \norm{f_0}_{L^2} \,,
  \end{align*}
  provided $\delta' = 1 - \delta \in (0, 1)$.
  Using~\eqref{e:TDisTStar}, this implies
  \begin{equation}
    \tdis(\delta) \geq N
    \,.
  \end{equation}
  proving~\eqref{e:plDisTime} as desired.
\end{proof}

\appendix
\section{Double exponential convergence when \texorpdfstring{$\varphi$}{phi} is uniformly expanding, or an ergodic toral automorphism.}\label{s:dexp}

Throughout this appendix we will assume~$K_\epsilon$ is the periodized rescaled standard Gaussian.
That is, we assume~$\check K$ is a standard Gaussian and~$K_\epsilon$ is defined by~\eqref{e:KepsPeriodized}.
At the expense of a more technical proof, similar results can be obtained using only spectral assumptions on~$K_\epsilon$; we refer to~\cite{FannjiangNonnenmacherEA04} for details.

\begin{proposition}\label{p:UnifExpanding}
  Let $N \in \set{ 2, 3, \dots}$, and suppose~$\varphi \colon \T^d \to \T^d$ is the uniformly expanding map
  \begin{equation}\label{e:UnifExpanding}
    \varphi(x) = N x \pmod{\Z^d}\,.
  \end{equation}
  Then there exists a constant~$C$ such that for all~$n \in \N$ and~$\epsilon > 0 $ sufficiently small we have
  \begin{gather}\label{e:TVDexp}
    \norm{\dist(X_{n+1}) - \pi }_\TV \leq \frac{C}{\epsilon^{d/2}}
      \exp\paren[\Big]{\frac{- \epsilon^2 N^{2n}}{C}}
  \end{gather}
  Moreover, if~$T_*$ and~$T$ are the evolution operators defined in~\eqref{e:TStarDef} and~\eqref{e:TDef} respectively, then
  \begin{equation}\label{e:L2Dexp}
    \norm{T^n}_{\dot L^2 \to \dot L^2}
      = \norm{T_*^n}_{\dot L^2 \to \dot L^2}
      \leq\exp\paren[\Big]{ - \frac{\epsilon^2 N^{2n}}{C} }\,,
  \end{equation}
\end{proposition}
As an immediate corollary, we obtain bounds on the mixing time and dissipation time.
\begin{corollary}
  For the uniformly expanding map~\eqref{e:UnifExpanding}, there exists an explicit constant~$C = C(N, d)$ such that the mixing time and dissipation time are bounded by
  \begin{gather}
    \label{e:TMixUnif}
    \tmix(\delta) \leq
      \abs{\log_N \epsilon}
      + \frac{1}{2} \log_N\paren[\Big]{  \frac{d}{2} \abs{\ln \epsilon} + \abs{\ln \delta} + C}
      + C
  \\
    \label{e:TDisUnif}
    \tdis(\delta) \leq \abs{\log_N \epsilon} + \frac{1}{2} \log_N \abs{\ln \delta} + C
  \end{gather}
\end{corollary}
\begin{remark}
  The uniformly expanding map~\eqref{e:UnifExpanding} is of the form considered in Section~\ref{s:Phi}.
  Indeed, for the uniformly expanding map,~$M = N^d$, and each~$E_i$ is a cube of side length~$1/N$ and $p_{\max} = p_{\min} = 1/N^d$.
  From this we see~\eqref{e:TMixUnif} attains the conjectured upper bound~\eqref{e:TMixDeltaSharpUpper} (with an explicit~$\delta$ dependence), and~\eqref{e:TDisUnif} improves the~$\delta$ dependence in~\eqref{e:plDisTime} by replacing $\abs{\ln \delta} / \delta^2$ with a double logarithm.
\end{remark}
\begin{remark}
  If~$\varphi$ is an ergodic toral automorphism (with matrix~$A$), then the double exponential bounds in Proposition~\ref{p:UnifExpanding} still hold.
  The proof is contained in~\cite{FannjiangWoowski03,FengIyer19}, and the main idea is as follows.
  One can use certain Diophantine approximation results to show that there exists $\lambda > 1$ such that 
  \begin{equation}
    \abs{A_*^{-n} k} \geq \frac{\lambda^n}{C \abs{k}^{d-1}}\,,
    \quad \text{for all}\quad k \in \Z^d - \set{0}\,.
  \end{equation}
  Here~$A_*$ is the transpose of~$A$, and $C$ is a dimensional constant.
  (The proof of this follows from Lemma 4.4 in~\cite{FengIyer19}, and is the lower bound of $|B^n k|$ in the proof of Proposition 4.1.)
  Once this is established, we can follow the proof of Proposition~\ref{p:UnifExpanding}.
  (The details are carried out in the proof of Theorem~2.12 in~\cite{FengIyer19}, and a slightly different proof is in~\cite{FannjiangWoowski03}.)
\end{remark}

\begin{proof}[Proof of Proposition~\ref{p:UnifExpanding}]
  Note first for any function~$g$ we have
  \begin{equation}
    Tg = (K_\epsilon * g) \circ \varphi\,,
  \end{equation}
  from which the Fourier coefficients of~$Tg$ can be computed explicitly.
  Namely, if $k' \notin N \Z^d$, then $(Tg)^\wedge(k') = 0$, and otherwise
  \begin{equation}\label{e:TgHatNk}
    (Tg)^\wedge( N k ) = \hat K_\epsilon(k) \hat g(k)
    \,.
  \end{equation}
  Since~$K_\epsilon$ is a periodized rescaled Gaussian, we know
  \begin{equation}
    \hat K_\epsilon(k) = e^{ -2\pi^2 \epsilon^2 \abs{k}^2 }\,.
  \end{equation}
  Thus iterating~\eqref{e:TgHatNk} gives
  \begin{equation}\label{e:TgHatNkIterated}
    (T^n g)^\wedge(N^n k)
      = \exp\paren[\Big]{ - 2\pi^2 \epsilon^2 \paren[\Big]{ \frac{N^{2n} - 1}{N - 1} } \abs{k}^2 }
	\hat g(k)\,,
  \end{equation}
  and $(T^n g)^\wedge(k') = 0$ when $k' \notin N^n \Z$.

  Now if~$g \in \dot L^2$, then $\hat g(0) = 0$ and Parseval's identity implies
  \begin{align}
    \norm{T^n g}_{L^2}
      &\leq
	\norm{g}_{L^2} \sup_{k \neq 0} 
	  \exp\paren[\Big]{
	    - 2\pi^2 \epsilon^2 \paren[\Big]{ \frac{N^{2n} - 1}{N - 1} } \abs{k}^2
	  }
    \\
      &\leq
	\norm{g}_{L^2} \exp\paren[\Big]{- 2\pi^2 \epsilon^2 \paren[\Big]{ \frac{N^{2n} - 1}{N - 1} } }\,.
  \end{align}
  Since~$\norm{T^n}_{\dot L^2 \to \dot L^2} = \norm{T_*^n}_{\dot L^2 \to \dot L^2}$, this implies~\eqref{e:L2Dexp}.
  \medskip

  To prove~\eqref{e:TVDexp}, suppose~$X_0 \sim \mu$ and observe
  \begin{align}
    \norm{\dist(X_{n+1}) - \pi}_\TV
      &= \frac{\norm{T_*^{n+1} \mu - 1}_{L^1}}{2}
      \leq \frac{\norm{T_*^{n+1} \mu - 1}_{L^2}}{2}
  \\
      &\leq \frac{\norm{T_*^n}_{L^2 \to L^2}  \norm{T_* \mu - 1}_{L^2}}{2}
  \\
      &\leq \norm{K_\epsilon}_{L^2} \exp\paren[\Big]{ -\frac{\epsilon^2 N^{2n}}{C} }
    \leq \frac{C}{\epsilon^{d/2}} \exp\paren[\Big]{ -\frac{\epsilon^2 N^{2n}}{C} }\,,
  \end{align}
  concluding the proof.
\end{proof}

\section*{Acknowledgements}
We gratefully acknowledge helpful discussions with Albert Fannjiang.

\bibliographystyle{halpha-abbrv}
\bibliography{refs,preprints}
\end{document}